\documentclass{amsart}

\usepackage[ngerman, english]{babel}
\usepackage[utf8]{inputenc}
\usepackage{hyperref}

\usepackage{amssymb}
\usepackage{amsmath}
\usepackage{amsthm}
\usepackage{bbm}
\usepackage{dsfont}
\usepackage{hyphenat}

\usepackage[shortlabels]{enumitem}
\usepackage{tikz}
\usepackage{marginnote}
\usepackage{mathtools}
\usepackage{nicefrac}
\usepackage{appendix}
\usepackage{orcidlink}

\newcommand{\bbK}{\mathbb{K}}

\newcommand{\bbN}{\mathbb{N}}

\newcommand{\bbR}{\mathbb{R}}

\newcommand{\calB}{\mathcal{B}}

\newcommand{\calF}{\mathcal{F}}

\newcommand{\calL}{\mathcal{L}}
\newcommand{\calM}{\mathcal{M}}

\newcommand{\calP}{\mathcal{P}}

\newcommand{\calS}{\mathcal{S}}

\newcommand{\calU}{\mathcal{U}}

\newcommand{\eps}{\varepsilon}
\DeclareMathOperator{\id}{id} % identity operator
\DeclareMathOperator{\one}{\mathbbm{1}} % constant function with value 1 / indicator functions
\DeclareMathOperator{\re}{Re} % real part
\newcommand{\argument}{\mathord{\,\cdot\,}} % argument operator
\newcommand{\ud}{\mathrm{d}}
\DeclareMathOperator{\Ball}{B}
\DeclarePairedDelimiter{\norm}{\lVert}{\rVert} % norm of a vector (\norm does not scale, \norm* is the left-right-option, \norm[\bigg] ,e.g., for other sizes (\bigg in this case)) 
\DeclarePairedDelimiter{\modulus}{\lvert}{\rvert} % modulus
\DeclarePairedDelimiter{\dual}{\langle}{\rangle} % dual pairings
\DeclarePairedDelimiter{\set}{\lbrace}{\rbrace} % sets
\newcommand{\fm}{\mathfrak{m}}

\newcommand{\impliesProof}[2]{\ref{#1} $\Rightarrow$ \ref{#2}: }

\newcommand{\N}{\bbN}

\newcommand{\R}{\bbR}

\newcommand{\K}{\bbK}
\renewcommand{\subset}{\subseteq}

\theoremstyle{definition}
\newtheorem{defn}{Definition}[section]
\newtheorem{rem}[defn]{Remark}

\newtheorem{example}[defn]{Example}

\newtheorem{hypo}[defn]{Hypothesis}

\theoremstyle{plain}
\newtheorem{prop}[defn]{Proposition}
\newtheorem{lem}[defn]{Lemma}
\newtheorem{thm}[defn]{Theorem}
\newtheorem{cor}[defn]{Corollary}

\numberwithin{equation}{section}
\begin{document}
\title{Ultra Feller operators from a functional analytic perspective}
\author{Alexander Dobrick
\orcidlink{0000-0002-3308-3581}}
\address[A.~Dobrick]{Christian-Albrechts-Universit\"at~zu~Kiel, Arbeitsbereich~Analysis, Hein-rich-Hecht-Platz~6, 24118~Kiel, Germany}
\email{dobrick@math.uni-kiel.de}
\author{Julian H\"olz
\orcidlink{0000-0001-5058-9210}}
\address[J.\ H\"olz]{Bergische~Universit\"at~Wuppertal, Fakult\"at~f\"ur~Mathematik~und~Naturwissenschaften, Gaußstr.~20, 42119~Wuppertal, Germany}
\email{hoelz@uni-wuppertal.de}
\author{Markus Kunze
\orcidlink{0000-0001-5856-7382}}
\address[M.\ Kunze]{Universit\"at~Konstanz, Fachbereich~Mathematik~und~Statistik, Fach~193, 78357~Konstanz, Germany}
\email{markus.kunze@uni-konstanz.de}

\date{\today}
\begin{abstract}
    It is a widely acknowledged fact that the product of two positive strong Feller operators on a Polish space $E$ enjoys the ultra Feller property. 
    We present a functional analytic proof of this fact that allows us to drop the assumption that the operators are positive, and also extends the applicability of this result to more general state spaces.
    As it turns out, this result can be considered a variant of the theorem that on a Banach space with the Dunford--Pettis property, the product of two weakly compact operators is compact.
\end{abstract}

\maketitle

\section{Introduction}
\label{section:introduction}

\subsection*{Strong and ultra Feller properties}

Let $E$ be a Polish space and denote by $C_b(E)$ the continuous and bounded functions $f: E \to \bbK$, where $\bbK$ denotes either the real or the complex field. A kernel operator on $C_b(E)$ is a bounded linear operator $T \in \calL(C_b(E))$ that is given by
\begin{equation}\label{eq.kop}
    [Tf](x) = \int_E f(y) \, k(x, \ud y),
\end{equation}
where $k$ is a kernel (see Section~\ref{sect.classic}). We point out that the requirement that $T$ maps $C_b(E)$ to $C_b(E)$ implies that the measure-valued map $x\mapsto k(x, \argument)$ is continuous with respect to the $\sigma(\calM(E), C_b(E))$-topology, where $\calM(E)$ denotes the space of bounded, signed or complex Radon measures on the Borel $\sigma$-algebra of $E$. 
It should also be noted that in many applications (e.g. in probability theory) one is mainly interested in \emph{positive} kernel operators, which means that $k(x,A) \ge 0$ for all $x \in E$ and all measurable $A \subseteq E$.

In many instances, the kernel $k$ satisfies even stronger continuity properties. Following~\cite[Definition~5.7]{Revuz1984}, we say that the kernel $k$ (or the operator $T$) satisfies the \emph{strong Feller property}, if the map $x\mapsto k(x, \argument)$ is continuous in the $\sigma(\calM(E), B_b(E))$-topology, where $B_b(E)$ denotes the space of all bounded, Borel measurable functions on $E$. Moreover, it enjoys the \emph{ultra Feller property} if $x \mapsto k(x, \argument)$ is continuous with respect to the total variation norm on $\calM(E)$.

It is well known that on a Polish space $E$, a positive kernel operator $T$ enjoys the ultra Feller property if and only if it maps bounded subsets of $C_b(E)$ to sets which are relatively compact in the topology of uniform convergence on compact subsets of $E$, see~\cite[Proposition~1.5.8]{Revuz1984}. It is this compactness property which is of great interest in many applications. Nevertheless, one usually does not assume the ultra Feller property but rather the (weaker) strong Feller property. This is due to the fact that the product of two positive strong Feller operators is an ultra Feller operator (see~\cite[Theorem~1.5.9]{Revuz1984}). 
Thus, as one is often not interested in a single operator $T$, but rather in a \emph{semigroup} ${(T(t))}_{t \geq 0}$ of operators, the semigroup law implies that $T(t)$ is a strong Feller operator for all $t > 0$ if and only if it is an ultra Feller operator for all $t > 0$.

The strong Feller assumption appears in particular in the study of asymptotic behaviour of stochastic processes and transition semigroups, see for instance~\cite[Section~11.2.3]{DaPrato1992} for a classical reference in stochastic analysis and~\cite{Gerlach2014, Gerlach2023, Gerlach2012} for recent contributions from the perspective of operator theory.
The strong Feller property was also used in semigroup theory for approximation~\cite{Budde2023} and perturbation~\cite{Kuehn2022, Kunze2013} results. In~\cite{Benaim2022}, the strong Feller property plays an important role in a criterion for quasi-compactness of an operator. There has been speculation regarding whether the strong Feller property of a transition semigroup implies that the associated stochastic process can be realized with paths that are right-continuous with existing left limits (so called \emph{càdlàg paths}). This has been recently disproved in~\cite{Beznea2024, Beznea2011}.

\subsection*{Weakly compact operators and the Dunford--Pettis property} 

Let us briefly sketch the classical proof (see, e.g., \cite{Revuz1984}) that the product of two positive strong Feller operators $S$ and $T$ 
on a Polish space $E$ enjoys the ultra Feller property. Given a bounded sequence $(f_n)_n$, one typically proceeds in two steps:

\begin{enumerate}[\upshape 1.]
    \item
          One shows that one can find a subsequence $(f_{n_k})_k$ such that $S f_{n_k}$ converges pointwise (see~\cite[Lemma~1.5.10]{Revuz1984}). The main point in this step is to construct a measure $\mu$ on $E$ such that every measure $k(x, \argument)$, where $k$ denotes the kernel of $S$, is absolutely continuous with respect to this $\mu$. One can then appeal to the Banach--Alaoglu theorem in $L^\infty(E, \mu)$. We point out that this construction depends heavily on the assumption that $E$ is a Polish space.
    \item
          Subsequently, one uses the positivity of $T$ and Dini's theorem to show that if $(g_n)_n$ is a bounded sequence that converges pointwise to $g$, then $T g_n$ converges to $T g$ uniformly on compact subsets of $E$ (see~\cite[Lemma~1.5.11]{Revuz1984}).
\end{enumerate}

The second step shows that a positive strong Feller operator improves the mode of convergence and is reminiscent of the following result of Dunford and Pettis~\cite{Dunford1940}: 
If $T$ is a weakly compact operator on an $L^1$-space, it maps weakly convergent sequences to norm convergent sequences. This result motivated Grothendieck~\cite{Grothendieck1953} to call an operator between arbitrary Banach spaces $X$ and $Y$ a \emph{Dunford--Pettis operator} if it maps weakly convergent sequences to norm convergent sequences. Grothendieck also characterized Banach spaces $X$ where every weakly compact operator from $X$ to another Banach space $Y$ is a Dunford--Pettis operator; this is the case if and only if $X$ has the \emph{Dunford--Pettis property}. It follows immediately that if $X$ is a Banach space with the Dunford--Pettis property, then the product of two weakly compact operators on $X$ is compact.

\subsection*{Contributions of this article}

In this article, we extend the above two steps as follows:

As for the first step, we prove in Theorem~\ref{t.strongfellerop} that if $E$ is a completely regular Hausdorff topological space that is compactly generated, then an operator $T$ on $C_b(E)$ satisfies the strong Feller property if and only if it is $\sigma(C_b(E), \calM(E))$-compact. This generalizes
a result of Sentilles \cite{Sentilles1969, Sentilles1969a} who considered the case where $E$ is locally compact.

This result gives a functional analytic characterization of  
strong Feller operators and allows us to treat the second
step in a more general setup.
Namely, we replace the dual pair $(C_b(E), \calM(E))$ with 
a general norming dual pair $(X,Y)$ of Banach spaces
and instead of a strong Feller operator, we consider a $\sigma(X,Y)$-compact operator $T$ on $X$. In Theorem~\ref{t.dpimproving}, we prove that if $X$ has the Dunford--Pettis property
with respect to $Y$, see Definition~\ref{def.dp}, then such an operator $T$ maps a sequence that converges with respect to
$\sigma(X,Y)$ to a sequence that converges with respect to the Mackey topology $\mu(X,Y)$. This result is a generalisation of Grothendieck's result
to norming dual pairs. As $X=C_b(E)$ has the Dunford--Pettis property with respect to $Y=\calM(E)$ for every completely regular Hausdorff topological space
$E$, this also generalizes Step 2 above.

By combining these two results, we generalize the classical result that the product of strong Feller operators is ultra Feller in two directions (Theorem~\ref{t.productstrongfeller}): we show that positivity of the involved operators is not needed and, at the same time, we consider more general state spaces $E$.

As we have treated the second step in a general framework,
our abstract results apply to every norming dual pair $(X,Y)$ where
$X$ has the Dunford--Pettis property with respect to $Y$.
In Section \ref{sect.applications}, we discuss two examples
of this.
Namely, in Section~\ref{ssect.a1}, we consider the case $X=C_b(E)$ and $Y= L^1(\fm)$, where
$\fm$ is a Borel measure on $E$ satisfying suitable assumptions. As it turns out, in this case 
$\sigma(C_b(E), L^1(\fm))$-compactness of an operator $T$ is actually equivalent to \linebreak $\sigma(C_b(E), \calM(E))$-compactness, which is formally a stronger condition (see Theorem~\ref{t.cbl1sf}). In Section~\ref{ssect.a2}, we consider the case where
$E$ is locally compact and work on the norming dual pair 
$(C_0(E), \calM(E))$. In this case, $\calM(E)$ is the Banach space dual of $C_0(E)$, so we are back in the situation of Grothendieck's 
classical result. However, in this case we obtain an equivalent characterization of compact and weakly compact operators on $C_0(E)$ 
as kernel operators, see Theorem~\ref{t.compactc0}. 
This should be compared to~\cite[Section~5.3]{Ryan2002}, where a characterization of such operators via vector measures is given.

\section{Weak compactness of operators on norming dual pairs} \label{sect.weakcompactness}

\subsection{Norming dual pairs} 

We start this section by introducing the  fundamental notion of norming dual pairs. Here, $\bbK$ denotes either the real or the complex field.

\begin{defn}
\label{def.norming_pair}
    A \emph{norming dual pair} (over $\bbK$) is a triple $(X, Y, \langle \cdot, \cdot\rangle)$ consisting of two non-trivial $\bbK$-Banach spaces $X$ and $Y$ and a $\bbK$-bilinear mapping 
    \begin{align*}
        \dual{\cdot \,, \cdot} \colon X \times Y \to \bbK,
    \end{align*}
    called \emph{duality pairing between $X$ and $Y$}, such that
    \[
        \|x\| = \sup \{ |\langle x, y\rangle| : \|y\| \leq 1\} \quad \text{and} \quad \|y\| = \sup \{ |\langle x,y\rangle| : \|x\|\leq 1\}.
    \]
    We often omit $\langle\cdot \,, \cdot\rangle$ and merely say that $(X,Y)$ is a norming dual pair. To eliminate potential ambiguities, we sometimes specify the bilinear mapping by writing $\langle \cdot\,, \cdot \rangle_{X,Y}$ instead of $\langle \cdot\,, \cdot \rangle$. In what follows, we write $\Ball_Z$
    for the closed unit ball of a Banach space $Z$.
\end{defn}

Due to the symmetry of the definition, $(X, Y)$ is a norming dual pair if and only if $(Y, X)$ is a norming dual pair. Moreover, if $(X, Y)$ is a norming dual pair, then $Y$ can be identified with a closed subspace of the norm dual $X^\ast$ of $X$ via the continuous embedding
\begin{align*}
    Y \to X^\ast, \quad y \mapsto \dual{\cdot \, ,y}.
\end{align*}
Conversely, if $X$ is a Banach space and $Y$ is a subspace of $X^\ast$, then $(X, Y)$ is a norming dual pair if and only if $Y$ is a closed, norming subspace of $X^\ast$. \smallskip

If $(X, Y)$ is a norming dual pair, then one can consider several standard locally convex topologies defined in terms of the duality (see \cite[Chapter~IV]{Schaefer1999}). In particular, on $X$ one might consider the topology $\sigma(X,Y)$ of pointwise convergence on $Y$ and the \emph{Mackey topology} $\mu(X,Y)$ of uniform convergence on absolutely convex, $\sigma(Y,X)$-compact subsets of $Y$.\smallskip

Recall that a locally convex topology $\tau$ on $X$ is called \emph{consistent} with the norming dual pair $(X,Y)$ if the dual space $(X, \tau)'$ coincides with $Y$,
i.e., the $\tau$-continuous linear functional are precisely those of the form $\langle \argument , y \rangle$ with $y \in Y$. 
The Mackey--Arens theorem \cite[Theorem~IV.3.2]{Schaefer1999} implies that any locally convex topology that is consistent with $(X,Y)$ is finer than $\sigma(X,Y)$ and coarser than $\mu(X,Y)$.
In particular, the Mackey topology is the finest locally convex topology on $X$ which is consistent with 
the norming dual pair $(X, Y)$. 
\smallskip

The following proposition, which generalizes \cite[Lemma~A.1]{Gerlach2023}, is a direct consequence of the bipolar theorem and will be useful later on. 

\begin{prop} \label{p.dense}
    Let $(X, Y)$ be a norming dual pair and $Y_0 \subset Y$ be a closed subspace. Suppose that $(X, Y_0)$ is also a norming dual pair. Then the closed unit ball of $Y_0$ is $\sigma(Y,X)$-dense in the closed unit ball of $Y$.
\end{prop}

\begin{proof}
    The polar of $\Ball_{Y_0}$ in $X$ with respect to $Y$ is defined as
    \begin{align*}
        (\Ball_{Y_0})^\circ
        &=
        \{x \in X:  \re \langle x, y \rangle \le 1  \text{ for all } y \in \Ball_{Y_0} \} \\
        &= \{x \in X:  \lvert \langle x, y \rangle \rvert \le 1  \text{ for all } y \in \Ball_{Y_0} \},
    \end{align*}
    as the unit ball is balanced, so that the polar and the absolute polar coincide. 
    Since the dual pair $(X, Y_0)$ is norming, this set equals the closed unit ball $\Ball_X$ in $X$.
    Similarly, as $(X, Y)$ is norming, the polar of $\Ball_X$ in $Y$ with respect to $X$, which is defined as
    \begin{align*}
        (\Ball_X)^\circ
        &=
        \{y \in Y:  \re \langle x, y \rangle \le 1  \text{ for all } x \in \Ball_X \} \\
        &= \{y \in Y:  \lvert \langle x, y \rangle \rvert \le 1  \text{ for all } x \in \Ball_X \},
    \end{align*}
    equals the closed unit ball $\Ball_Y$ in $Y$. Hence, the bipolar $(\Ball_{Y_0})^{\circ\circ}$ of $\Ball_{Y_0}$ in the dual pair $(X, Y)$ is equal to $\Ball_Y$. But as $\Ball_{Y_0}$ is convex and contains $0$, the bipolar theorem (see e.g.~\cite[p.~126]{Schaefer1999}) implies that the bipolar $(\Ball_{Y_0})^{\circ\circ}$ is the $\sigma(Y,X)$-closure of $\Ball_{Y_0}$ in $Y$.
\end{proof}

\subsection{Weakly continuous and weakly compact operators}

Next, we consider continuous linear operators in the context of norming dual pairs. In what follows, we refer to a linear mapping between vector spaces as an \emph{operator}. On locally convex spaces, we often study operators continuous in specific topologies. We denote the space of \emph{norm-continuous operators on $X$} by $\calL(X)$ and, for a locally convex topology $\tau$ on $X$, we denote the space of \emph{$\tau$-continuous operators on $X$} by $\calL(X, \tau)$.  
If $T\in \calL(X)$, we denote its \emph{$\|\argument\|$-adjoint} by $T^*: X^* \to X^*$ whereas for $T\in \calL(X, \sigma(X,Y))$ we denote its $\sigma(X,Y)$-adjoint by $T': Y \to Y$, which is uniquely determined by the relation
\begin{align*}
    \langle Tx, y \rangle = \langle x, T'y \rangle
\end{align*}
for all $x \in X$ and all $y \in Y$.
Let us collect some results concerning the continuity of operators on norming dual pairs.

\begin{prop} \label{p.weaklycontinuous}
    Let $(X, Y)$ be a norming dual pair and let $T \colon X \to X$ be a linear map. Then the following assertions hold: 
    \begin{enumerate} 
    [\upshape (a)]
    \item A subset of $X$ is $\sigma(X,Y)$-bounded if and only if it is norm bounded. Consequently, if 
    $T$ is $\sigma(X,Y)$-continuous, then $T$ is $\|\argument\|$-continuous.
    \item $T$ is $\sigma(X,Y)$-continuous if and only if it is $\mu(X,Y)$-continuous if and only if it is 
    $\tau$-continuous for any locally convex topology that is consistent with the pairing $(X,Y)$.
    \item For an operator $T\in \calL(X)$, the following are equivalent:
    \begin{enumerate}[\upshape (i)]
        \item \label{i.p.weaklycontinuous.TWeaklycontinuous}
            $T \in \calL(X, \sigma(X, Y))$.
        \item \label{i.p.weaklycontinuous.TStarLeavesYInvariant} 
            $T \in \calL(X)$ and $T^\ast Y \subset Y$. 
    \end{enumerate}
    In this case, $T' = T^\ast \rvert_{Y}$ and $\norm{T}_{\calL(X)} = \norm{T'}_{\calL(Y)}$. 
    \end{enumerate}
\end{prop}

\begin{proof}
    (a): Clearly, a norm bounded subset is $\sigma(X,Y)$-bounded. Since $Y$ is norming for $X$, the converse follows from the uniform boundedness principle.

    (b): This is \cite[IV.7.4]{Schaefer1999}.

    (c): This is \cite[Proposition 3.1(i)]{Kunze2011}.
\end{proof}

We continue by considering various weak compactness properties for sets and operators.
Recall that a point $x$ in a topological space $\Omega$ is called an \emph{accumulation point} or \emph{limit point} of a set $A\subseteq \Omega$ 
if every neighbourhood of $x$ intersects $A \setminus \{x\}$.
A subset $C$ of a topological space $\Omega$ is called \emph{(relatively) sequentially compact} if every sequence in $C$ has a subsequence that converges to a point in $C$ (in $\Omega$). Similarly, $C$ is called \emph{(relatively) countably compact} if every sequence in $C$ has a subnet that converges to a point in $C$ (in $\Omega$); it is not difficult to see that $C$ is (relatively) countably compact if and only if every  infinite subset of $C$ has an accumulation point in $C$ (in $\Omega$).
Note that relative countable/sequential compactness of a set does not imply countable/sequential compactness of its closure, see, e.g.,~\cite[Exercise 2.15~on~p.~161]{Megginson1998}.

Using the above notions, we define compactness properties for $\sigma(X,Y)$-continuous operators on norming dual pairs. We call an operator $T \in \calL(X, \sigma(X,Y))$ 
\emph{$\sigma(X,Y)$-compact} if it maps bounded subsets of $X$ to relatively $\sigma(X,Y)$-compact subsets of $X$.
Likewise, $T$ is called \emph{sequentially $\sigma(X,Y)$-compact} if it maps bounded subsets of $X$ to relatively sequentially $\sigma(X,Y)$-compact subsets.
If $T$ maps bounded subsets of $X$ to 
relatively countably $\sigma(X,Y)$-compact subsets, then $T$
is called \emph{countably $\sigma(X,Y)$-compact}.

Gantmacher~\cite{Gantmacher1940} provided a characterization of weakly compact operators on Banach spaces (see also \cite[Theorem~5.23]{Aliprantis2006a}). This result can
be generalized to topological vector spaces \cite[Proposition~II.9.4]{Schaefer1974}. 
In the context of norming dual pairs, one can prove the following version of Gantmacher's result.

\begin{prop} \label{p.weakcomp}
    Let $(X,Y)$ be a norming dual pair and $T \in \calL(X,\sigma(X,Y))$. Set $S \coloneqq T' \in \calL(Y,\sigma(Y,X))$. Then the following assertions are equivalent:
    \begin{enumerate}[\upshape (i)]
        \item \label{i.p.weakcomp.TWeakComp}
            $T$ is $\sigma(X,Y)$-compact.
        \item \label{i.p.weakcomp.SLeavesYInvariant}
            $S^*Y^*\subset X$.
        \item \label{i.p.weakcomp.SContinuousToNorm}
            $S$ is continuous as an operator from $(Y, \mu(Y,X))$ to $(Y, \| \argument \|_Y)$.
        \item \label{i.p.weakcomp.SContinuousToWeak}
            $S$ is continuous as an operator from $(Y, \mu(Y,X))$ to $(Y, \sigma(Y, Y^*))$.
    \end{enumerate}
\end{prop}

\begin{proof}
    Before beginning, we note that $S^*$ is a $\sigma(Y^*, Y)$-continuous operator that extends $T$. 

    \impliesProof{i.p.weakcomp.TWeakComp}{i.p.weakcomp.SLeavesYInvariant}
        By Proposition~\ref{p.dense}, applied to the norming dual pairs $(Y^*,Y)$ and $(X,Y)$, 
        $\Ball_X$ is $\sigma(Y^*,Y)$-dense in $\Ball_{Y^*}$. Thus, for each $y^* \in Y^*$, there exists a net $(x_\lambda)_\lambda$ in $\Ball_X$ that converges to $y^*$ with respect to $\sigma(Y^*,Y)$. By (i), there is a subnet $(x_\mu)_\mu$ such that the net $(T x_\mu)_\mu$ converges to some $x \in X$ with respect to $\sigma(X, Y)$. It follows that
        \[
            S^*y^* = \lim_\mu S^*x_\mu = \lim_\mu Tx_\mu = x \in X.
        \]

    \impliesProof{i.p.weakcomp.SLeavesYInvariant}{i.p.weakcomp.SContinuousToNorm}
        Clearly, $S^*: Y^* \to X$ is $\sigma(Y^*,Y)$-to-$\sigma(X,Y)$-continuous. As $\Ball_{Y^*}$ is $\sigma(Y^*,Y)$-compact
        by the Banach--Alaoglu theorem, the set $C \coloneqq S^*\Ball_{Y^*}$ is $\sigma(X,Y)$-compact.
        
        It follows that the polar $C^\circ$ is a neighbourhood of $0$ for the $\mu(Y,X)$-topology.
        As $C$ is balanced, $C^\circ$ coincides with its absolute polar. Thus, $y\in C^\circ$ if and only if
        \[
        |\langle y^*, Sy\rangle_{Y^*,Y}| = |\langle S^*y^*, y\rangle_{X,Y}| \leq 1
        \]
        for all $y^*\in \Ball_{Y^*}$. This implies $\|Sy\|\leq 1$ for all $y\in C^\circ$, whence
        $SC^\circ \subset \Ball_Y$, proving (iii). \smallskip

    \impliesProof{i.p.weakcomp.SContinuousToNorm}{i.p.weakcomp.SContinuousToWeak}
        This is an immediate consequence of the continuity of the embedding $(Y, \|\argument\|_Y) \hookrightarrow (Y, \sigma(Y, Y^*))$. \smallskip

    \impliesProof{i.p.weakcomp.SContinuousToWeak}{i.p.weakcomp.SLeavesYInvariant}
        Let $y^* \in Y^*$. Then~(iv) implies that the functional $S^* y^* \in Y^*$ is a continuous map from $(Y, \mu(Y,X))$ to the scalar field, so $S^*y^*$ is in the topological dual space of $(Y, \mu(Y,X))$, which is $X$. \smallskip

    \impliesProof{i.p.weakcomp.SLeavesYInvariant}{i.p.weakcomp.TWeakComp}
        The set $C$ from the proof of the implication (ii) $\Rightarrow$ (iii) is $\sigma(X,Y)$-compact and contains $T\Ball_X$. This proves (i).
\end{proof}

\begin{rem}
    Among other things, Gantmacher proved in \cite{Gantmacher1940}
    that a bounded operator on a Banach space $X$ is $\sigma(X, X^*)$-compact 
    if and only if its adjoint $T^*$ is $\sigma(X^*, X^{**})$-compact. 
    Nowadays, this result is often referred to as Gantmacher's theorem. 
    In the setting of norming dual pairs, this equivalence is not true in general, 
    i.e., we can have a $\sigma(X,Y)$-compact operator on $X$ whose adjoint $T'$ is not $\sigma(Y, X)$-compact. 
    
    This is easily seen by considering the norming dual pair $(Y^*, Y)$, i.e., choosing $X = Y^*$. In this situation, every $\sigma(Y^*, Y)$-continuous operator $T$ on $X$ is $\sigma(Y^*, Y)$-compact, but its $\sigma(Y^*,Y)$-adjoint $S$ is not necessarily $\sigma(Y,Y^*)$-compact: For instance, one can consider a non-reflexive Banach space $Y$ and let $T$ be the identity operator on $X=Y^*$. 
\end{rem}

In what follows, it will be crucial to determine in which situations the notions of compactness, sequential compactness and countable compactness of an operator coincide. For convenience, we recall the classical result of Eberlein--\v{Smulian} which asserts that this is the case for 
Banach spaces $X$ endowed with the weak topology $\sigma(X, X^*)$.

\begin{prop}
    \label{p.es}
    Let $X$ be a Banach space and let $A \subseteq X$. Then the following assertions are equivalent:
    \begin{enumerate}
        [\upshape (i)]
        \item $A$ is (relatively) $\sigma(X, X^*)$-compact;
        \item Every sequence in $A$ has a subsequence that $\sigma(X, X^*)$-converges to an element of $A$ (resp.\ an element of $X$);
        \item Every sequence in $A$ has a subnet that $\sigma(X, X^*)$-converges to an element of $A$ (resp.\ an element of $X$).
    \end{enumerate}
\end{prop}

\begin{proof}
    See \cite[Theorem 3.40]{Aliprantis2006a}.
\end{proof}

We recall that a topological space $\Omega$ is called \emph{angelic} if for every relatively countably compact subset $A$ of $\Omega$ the following 
hold: (i) $A$ is relatively compact and (ii) every point of $\overline{A}$ is the limit of a sequence contained in $A$.
The question whether a Banach space $X$ endowed with certain
weak topologies is angelic has received quite some attention (see \cite{Cascales1998} for the case $X = C(K)$, \cite{Cascales1995} for general Banach spaces,
where the topology is induced by a norming subset of the dual space).

We next present some generalizations of the Eberlein--\v{S}mulian result to locally convex space. 
We recall that a locally convex space $(X, \tau)$ is called \emph{quasi-complete} if every bounded and closed subset of $X$ is complete.

\begin{lem} \label{l.eberlein_smulian}
    Let $(X,Y)$ be a norming dual pair and $T \in \calL(X)$. Then the following assertions hold:
    \begin{enumerate}[\upshape (a)]
        \item \label{i.l.eberlein_smulian.smulian}
            Suppose that $Y$ is separable with respect to $\sigma(Y,X)$ (equivalently, that $Y$ contains a sequence that separates the points of $X$). Then the operator $T$ is sequentially $\sigma(X,Y)$-compact if and only if it is countably $\sigma(X,Y)$-compact if and only if it is
        $\sigma(X,Y)$-compact.
        
        \item \label{i.l.eberlein_smulian.eberlein}
            Suppose that $X$ equipped with the Mackey topology $\mu(X,Y)$ is quasi-complete and let $\sigma(X,Y) \subset \tau \subset \mu(X,Y)$ be a locally convex topology on $X$. Then $T$ is countably $\tau$-compact if and only if it is $\tau$-compact.
    \end{enumerate}
\end{lem}

\begin{proof}
\ref{i.l.eberlein_smulian.smulian} follows from a version of \v{S}mulian's theorem (see \cite[Theorem~in~\S~3.2]{Floret1980}) and~\ref{i.l.eberlein_smulian.eberlein} from a version of Eberlein's theorem (see \cite[Theorem~2~in~\S~24]{Koethe1969}).
\end{proof}

Consider a Banach space $X$ and let $Y$ be a closed, $\sigma(X^*,X)$-dense subspace of $X^*$. Then, by \cite[Proposition 2.6]{Guirao2019}, $(X,\mu(X,Y))$ is quasi-complete if and only if it is complete. Note, that if $(X,Y)$ is a norming dual pair, then $Y$ is $\sigma(X^*,X)$-dense in $X$,
so this result is applicable in our setting. We refer to \cite{Guirao2017} for further conditions ensuring completeness.
On the other hand, \cite[Example 3]{Bonet2010} provides an example, where $(X, \mu(X,Y))$ is not quasi-complete. 
Finally, we remark that by the Bourbaki--Robertson Lemma~\cite[\S18.4.4]{Koethe1969}, $(X, \mu(X,Y))$ is complete if and only if there exists a consistent topology $\tau$ on $X$ such that $(X, \tau)$ is complete.

\begin{prop} \label{p.choquet}
    Let $(X,Y)$ be a norming dual pair and let $K \subset Y$ be non-empty, convex and $\sigma(Y,X)$-compact. Then the operator
    \begin{align*}
        \Phi \colon X \to C(K), \quad (\Phi x)(y) \coloneqq \langle x, y \rangle_{X,Y}
    \end{align*}
    has the following properties:
    \begin{enumerate}[\upshape (a)]
        \item \label{i.p.choquet.imageNormAdjoint}
            The range of the norm adjoint $\Phi^*\colon C(K)^* \to X^*$ is contained in $Y$.

        \item \label{i.p.choquet.PhiContinuousWeak}
            $\Phi$ is continuous with respect to the topologies $\sigma(X,Y)$ and $\sigma(C(K), C(K)^*)$.

        \item \label{i.p.choquet.PhiContinuousMackey}
              $\Phi$ is continuous with respect to the Mackey topology $\mu(X,Y)$ and the supremum norm on $C(K)$.
    \end{enumerate}
\end{prop}

\begin{proof}
    \ref{i.p.choquet.imageNormAdjoint}: Clearly, $\Phi \colon X \to C(K)$ is a bounded operator. For each $y \in K$, let $\delta_y \in C(K)^*$ be the Dirac measure at $y$. Moreover, let $C \subseteq C(K)^*$ be the convex hull of the set $\set{\delta_y : y \in K}$. Then $C$ is $\sigma(C(K)^*, C(K))$-dense in the set of Radon probability measures $\calP(K)$ on $K$ by the Hahn--Banach separation theorem. Furthermore, one has $\Phi^* \delta_y = y$ for all $y \in K$ and thus
    \begin{align}
        \Phi^*(\calP(K)) = K \subseteq Y, \label{eq.range-phi-star}
    \end{align}
    since $K$ is convex and $\Phi^*$ is continuous with respect to $\sigma(C(K)^*, C(K))$ and $\sigma(X^*, X)$. 
    From this, it follows that $\Phi^*(C(K)^\ast)\subseteq Y$ by decomposing
    a general measure $\mu \in C(K)^*$ first into real and imaginary
    parts and then those into positive and negative parts. \smallskip
    
    \ref{i.p.choquet.PhiContinuousWeak}:
    This is a straightforward consequence of~\ref{i.p.choquet.imageNormAdjoint}. \smallskip

    \ref{i.p.choquet.PhiContinuousMackey}:
    Suppose that $(x_\lambda)_\lambda$ is a net in $X$ converging to $x \in X$ with respect to the Mackey topology $\mu(X, Y)$. From the joint continuity with respect to $\sigma(Y, X)$ of the scalar multiplication, it follows that the absolute convex hull
    \begin{align*}
        C \coloneqq \{\alpha z : \alpha \in \bbK, \, \modulus{\alpha} \leq 1, \, z \in K\}
    \end{align*}
    of $K$ is also $\sigma(Y, X)$-compact. Hence, it follows from the definition of the Mackey topology that
    \begin{align*}
        \norm{\Phi(x_\lambda) - \Phi(x)}_\infty = \sup_{y \in K} \modulus{\langle x_\lambda - x, y \rangle} \leq \sup_{y \in C} \modulus{\langle x_\lambda - x, y \rangle} \to 0.
    \end{align*}
    This yields the $\mu(X, Y)$-$\norm{\argument}_\infty$-continuity of $\Phi$.
\end{proof}

\begin{lem}
    \label{l.compact-improving}
    Let $(X,Y)$ be a norming dual pair and let $T\in \calL(X,\sigma(X,Y))$. Set $S \coloneqq T' \in \calL(Y,\sigma(Y,X))$. If $T$ is countably $\sigma(X,Y)$-compact, then for every convex and $\sigma(Y,X)$-compact subset $K \subset Y$ the image $S(K)$ is $\sigma(Y,Y^*)$-compact.
\end{lem}

\begin{proof}
    Consider the map $\Phi$ from Proposition~\ref{p.choquet}. The continuity of $\Phi$ proved in the proposition and the compactness assumption on $T$ imply that the composition
    \begin{align*}
        X \overset{T}{\longrightarrow} X \overset{\Phi}{\longrightarrow} C(K)
    \end{align*}
    is countably $\sigma(C(K), C(K)^*)$-compact. It follows from Proposition \ref{p.es}, that 
    $\Phi T$ is even $\sigma (C(K),C(K)^*)$-compact. Thus, the classical theorem of Gantmacher \cite[Theorem~5.23]{Aliprantis2006a} implies that $(\Phi T)^*\colon C(K)^* \to X^*$ is $\sigma(X^*,X^{**})$-compact. 

    Moreover, \eqref{eq.range-phi-star} implies that $S(K) = T^*(K) = (\Phi T)^*(\calP(K))$, where $\calP(K)$ again denotes the set of Radon probability measures on $K$. In particular, $S(K)$ is relatively weakly compact in $X^*$. As $S \in \calL(Y, \sigma(Y, X))$, the set $S(K)$ is $\sigma(Y, X)$-compact and thus $\sigma(Y, X)$-closed. In particular, $S(K)$ is $\sigma(Y, Y^*)$-closed.
\end{proof}

\subsection{The Dunford--Pettis property on norming dual pairs}

Next we introduce an important notion for this paper, the Dunford--Pettis property in the framework of norming dual pairs.

\begin{defn}\label{def.dp}
    Let $(X,Y)$ be a norming dual pair. We say that $X$ has the \emph{Dunford--Pettis property with respect to $Y$}, if for every sequence $(x_n)_n$ in $X$ converging to $0$ with respect to $\sigma(X, Y)$ and every sequence $(y_n)_n$ in $Y$ converging to $0$ with respect to $\sigma(Y, Y^*)$, one has $\langle x_n, y_n\rangle \to 0$ as $n \to \infty$.
\end{defn}

\begin{rem}\label{rem.classical}
    In the classical situation (see \cite[Section~3.7]{MeyerNieberg1991}), a Banach space $X$ has the Dunford--Pettis property if and only if whenever $(x_n)_n$ is a sequence in $X$ converging to $0$ with respect to $\sigma (X,X^*)$ and $(x_n^*)_n$ is a sequence in $X^*$ which converges to $0$ with respect to $\sigma (X^*, X^{**})$, then $\langle x_n, x_n^*\rangle \to 0$ as $n \to \infty$ (see, e.g., \cite[Theorem~3.7.7]{MeyerNieberg1991}). Thus, $X$ has the classical Dunford--Pettis property if and only if $X$ has the Dunford--Pettis property with respect to $X^*$ in the sense of Definition~\ref{def.dp}. Examples of Banach spaces having the classical Dunford--Pettis property are given by $L^1(\Omega, \Sigma, \mu)$ or $C(K)$, where $(\Omega, \Sigma, \mu)$ is a measure space and $K$ a compact Hausdorff topological space. 
\end{rem}

In our next result, we use two notions from the theory of Banach lattices. Recall that the lattice operations on a Banach lattice $X$ are said to be 
\emph{sequentially $\sigma(X,Y)$-continuous} if for every sequence $(x_n)_n$ in $X$ converging to $0$ with respect to $\sigma(X, Y)$,
one has $\modulus{x_n} \to 0$ with respect to $\sigma(X, Y)$ as $n \to \infty$. Moreover, a Banach lattice $Y$ is called  \emph{AL-space} 
if its norm is additive on the positive cone, i.e., $\norm{y_1 + y_2}= \norm{y_1} + \norm{y_2}$ for all $y_1,y_2 \in Y_+$. 

\begin{prop}\label{p.ALDP}
    Let $(X,Y)$ be a norming dual pair, where $X$ is a Banach lattice with sequentially $\sigma(X,Y)$-continuous lattice operations and
    $Y$ is an AL-space. Then $X$ has the Dunford--Pettis property
    with respect to $Y$.
\end{prop}

\begin{proof}
    Let further $(x_n)_n$ be a sequence in $X$ converging to $0$ with respect to $\sigma(X, Y)$. By assumption, we have $\modulus{x_n} \to 0$ with respect to $\sigma(X, Y)$ as $n \to \infty$.

    Now let $(y_n)_n$ be a sequence in $Y$ that converges to $0$ with respect to $\sigma(Y, Y^*)$ as $n \to \infty$. 
    Then the set $\{y_n : n\in \N\}$ is a relatively weakly compact subset of the AL-space $Y$. 
    By the characterization of relatively weakly compact subsets of AL-spaces \cite[Theorem~2.5.4(iv)]{MeyerNieberg1991}, 
    for each $\eps > 0$, there exists $y \in Y_+$ such that
    \[
        \{y_n : n\in \N\} \subset [-y, y] + \eps \Ball_Y.
    \]
    It follows that
    \begin{align*}
        \limsup_{n \to \infty} |\langle x_n, y_n\rangle | & \leq \limsup_{n \to \infty} \langle |x_n|, |y_n|\rangle \\
        & \leq \limsup_{n \to \infty} \langle |x_n|, y\rangle + \eps \limsup_{n \to \infty}\|x_n\| \\
        & \leq \eps \sup \{ \|x_n\| : n \in \N\}.
    \end{align*}
    As $\eps>0$ was chosen arbitrarily, this shows that $\langle x_n, y_n\rangle \to 0$ as $n \to \infty$.
\end{proof}

\begin{rem}
    It is worth mentioning that in Proposition \ref{p.ALDP} 
    the assumption that $Y$ is an AL-space can be weakened
    to $Y$ being a Banach lattice with
    the \emph{positive Schur property}. The latter means
    that $\norm{y_n} \to 0$ as $n \to \infty$ for each weakly null sequence $(y_n)_n$ in $Y_+$. By \cite[Section 3]{Wnuk1989a},
    this notion is strictly weaker than being an AL-space. 
    The positive Schur property is closely related to the Dunford--Pettis theory. In particular, by \cite[Theorem]{Wnuk1989}, a Banach lattice $Y$ has the positive Schur property if and only if $Y$ is $\sigma$-Dedekind complete and if each operator $T \colon E \to c_0$ is a Dunford--Pettis operator if and only if it is regular. For more information on the positive Schur property, we refer to the references \cite{SanchezHenriquez1992, Wnuk1989, Wnuk1989a, Wnuk1992}. 
\end{rem}

We point out that the Dunford--Pettis property is not symmetric, i.e., in a norming dual pair $(X,Y)$ it may happen that $X$ has the Dunford--Pettis property with respect to $Y$ while $Y$ does not have the Dunford--Pettis property with respect to $X$. As a matter of fact, this might already happen in the case where $Y=X^*$ and under the additional assumptions of Proposition~\ref{p.ALDP}, as the following example illustrates.

\begin{example}
    Let $K \coloneqq [0, 1]$ and consider the norming dual pair $(C(K), \calM(K))$. As $C(K)$ is an $AM$-space and ${C(K)}^* = \calM(K)$, it follows from~\cite[Proposition~2.1.11]{MeyerNieberg1991} that the lattice operations in $C(K)$ are sequentially $\sigma (C(K), \calM(K))$-continuous. By Proposition~\ref{p.ALDP} $C(K)$ has the Dunford--Pettis property with respect to $\calM(K)$ (cf.~also Remark~\ref{rem.classical}). On the other hand, $\calM(K)$  does not have the Dunford--Pettis property with respect to $C(K)$. To see this, consider $\mu_n=\delta_{1/n}-\delta_0$. Then $\mu_n\to 0$ with respect to $\sigma(\calM(K), C(K))$. Moreover, let $(f_n)_n$ be a bounded sequence of functions in $C(K)$ that converges pointwise to $0$ and that satisfies $f_n(0) = 0$ and $f_n(1/n) = 1$ for all $n$. Then $f_n\to 0$ with respect to $\sigma(C(K), \calM(K))$, but $\langle \mu_n, f_n\rangle = 1 \not\to 0$ as $n \to \infty$.
\end{example}

The next result is the main result of this section. As mentioned in the introduction, it is a generalization of Grothendieck's result to norming dual pairs.

\begin{thm}\label{t.dpimproving}
    Let $(X,Y)$ be a norming dual pair such that $X$ has the Dunford--Pettis property with respect to $Y$. If $T\in \calL(X,\sigma(X,Y))$ is countably $\sigma(X,Y)$-compact, then $T$ is sequentially continuous from $\sigma(X,Y)$ to $\mu(X,Y)$.
\end{thm}

\begin{proof}
    Aiming for a contradiction, let $(x_n)_n$ be a sequence in $X$ such that $x_n \to 0$ with respect to $\sigma(X,Y)$ but $Tx_n \not\to 0$ with respect to the Mackey topology $\mu (X,Y)$ as $n \to \infty$.

    Recall that the Mackey topology $\mu(X,Y)$ is generated by the seminorms
    \begin{align*}
        p_K \colon X \to \R, \quad p_K(x) = \max_{y \in K} \lvert \langle x, y \rangle \rvert
        \qquad \text{for all } x \in X,
    \end{align*}
    where $K$ runs through all absolutely convex $\sigma(Y,X)$-compact subsets of $Y$. Hence, there exists such a set $K$ and some $\eps > 0$ such that, after passing to a subsequence, $p_K(Tx_n)\geq \eps$ for all $n \in \N$. By the definition of $p_K$, for each index $n \in \N$, there exists an element $y_n \in K$ such that $p(T x_n) = \lvert \langle x_n, y_n \rangle \rvert$.
    According to Lemma~\ref{l.compact-improving}, the set $T'(K)$ is $\sigma(Y,Y^*)$-compact and, hence, sequentially $\sigma(Y,Y^*)$-compact 
    by Proposition \ref{p.es}.
    So after passing to a subsequence, there exists some $y \in K$ such that $T'y_n \to T'y$ with respect to $\sigma(Y, Y^*)$ as $n \to \infty$. Thus,
    \[
        \eps \leq p_K(Tx_n)
        =
        \langle Tx_n, y_n \rangle_{X,Y}
        =
        \langle x_n, T'y\rangle _{X,Y} + \langle x_n, T'y_n - T'y \rangle_{X, Y} \to 0,
    \]
    as $X$ has the  Dunford--Pettis property with respect to $Y$;
    this is a contradiction.
\end{proof}

We are now in position to establish the following corollary that will serve as the abstract cornerstone in the concrete applications in the subsequent sections.

\begin{cor} \label{c.ultrafeller}
    Let $(X,Y)$ be a norming dual pair such that $X$ has the Dunford--Pettis
    property with respect to $Y$,
    let $S \in \calL(X, \sigma(X,Y))$ be sequentially $\sigma(X,Y)$-compact, and let $T \in \calL(X, \sigma(X,Y))$ be countably $\sigma(X,Y)$-compact. Then the following assertions hold: 
    \begin{enumerate}[\upshape (a)]
        \item \label{i.c.ultrafeller.TSSequentiallyMackeyCompact}
              The product $TS$ is sequentially $\mu(X,Y)$-compact.

        \item \label{i.c.ultrafeller.TSMackeyCompact}
              If $X$ is quasi-complete with respect to the Mackey topology $\mu(X,Y)$, then $TS$ is $\mu(X,Y)$-compact.
    \end{enumerate}
\end{cor}

\begin{proof}
    \ref{i.c.ultrafeller.TSSequentiallyMackeyCompact}:
    Let $(x_n)_n$ be a bounded sequence in $X$. Then there exists a subsequence $(x_{n_k})_k$ such that $Sx_{n_k}$ converges with respect to $\sigma(X,Y)$ to some $x\in X$ as $k \to \infty$. Thus, by Theorem~\ref{t.dpimproving}, $TSx_{n_k} \to z\coloneqq Tx$ with respect to $\mu(X,Y)$ as $k \to \infty$. \smallskip

    \ref{i.c.ultrafeller.TSMackeyCompact}:
    As sequential compactness implies countable compactness, this follows from Lemma~\ref{l.eberlein_smulian}\ref{i.l.eberlein_smulian.eberlein} applied to $\tau = \mu(X,Y)$.
\end{proof}

\begin{rem}
    A classical result of Dunford and Pettis states that if $X$ is a Banach space with the (classical) Dunford--Pettis property (cf.\ Remark~\ref{rem.classical}), then the square of a weakly compact operator is compact (see for example~\cite[Corollary~5.88]{Aliprantis2006a}). This result can be obtained from Corollary~\ref{c.ultrafeller} by applying it to the pair $(X, X^*)$ and observing that $\mu(X, X^*)$ coincides with the (complete) norm topology.
\end{rem}

\section{Strong and ultra Feller operators on the pair \texorpdfstring{$(C_b(E), \calM(E))$}{(Cb(E), M(E))}}\label{sect.classic}

Throughout this section, let $E$ be a completely regular Hausdorff topological space and denote its Borel $\sigma$-algebra by $\calB(E)$. The space of bounded and continuous scalar-valued functions on $E$ and the space of bounded and Borel measurable scalar-valued functions on $E$ are denoted by $C_b(E)$ and $B_b(E)$, respectively. Equipping these spaces with the supremum norm $\Vert \argument \Vert_\infty$ renders them both Banach spaces.

Recall that a Borel measure $\mu$ is called a \emph{Radon measure} if for every $A\in \calB(E)$ and $\eps>0$ there exists a compact set $K\subset A$ such that $|\mu|(A\setminus K) \leq \eps$. Here, $|\mu|$ denotes the total variation of the measure $\mu$. Further, we denote the space of bounded (signed or complex) Radon measures on the measurable space $(E, \calB(E))$ by $\calM(E)$. Note that $\calM(E)$ is a Banach space when equipped with the total variation norm $\lVert \argument \rVert_{\mathrm{TV}}$. For $f\in C_b(E)$ and $\mu \in \calM(E)$, we write
\begin{equation}\label{eq.duality}
    \langle f, \mu \rangle \coloneqq \int_E f\, \ud\mu.
\end{equation}

Besides the norm topology, we consider another locally convex topology on $C_b(E)$, namely the \emph{strict topology} $\beta_0(E)$, that is defined as follows: Let $\calF_0(E)$ be the space of bounded functions $\varphi \colon E \to \R$ that \emph{vanish at infinity}, i.e., for each $\eps > 0$ there exists a  compact set $K \subset E$ such that $|\varphi (x)|\leq \eps$ for all $x\in E \setminus K$. The topology $\beta_0(E)$ is generated by the seminorms ${(p_\varphi)}_{\varphi \in \calF_0(E)}$, where $p_\varphi(f) \coloneqq \|\varphi f\|_\infty$.

We emphasize that the strict topology $\beta_0(E)$, in general, does not coincide with the Mackey topology of the norming dual pair $(C_b(E), \calM(E))$. However, by \cite[Theorems~4.5~\&~5.8]{Sentilles1972}, this is the case if $E$ is a \emph{Polish space}, i.e., its topology is metrizable through a complete, separable metric. If $E$ is a compact Hausdorff topological space, both the Mackey topology $\mu(C_b(E), \calM(E))$ and the strict topology $\beta_0(E)$ coincide with the norm topology.

Our first lemma collects some information about the pair $(C_b(E), \calM(E))$ and demonstrates that our main results from Section~\ref{sect.weakcompactness} are applicable to the norming dual pair $(C_b(E), \calM(E))$. We recall that $E$ is \emph{compactly generated} (or a \emph{k-space}) if a set $U \subset E$ is open in $E$ if and only if $U \cap K$ is open in $K$ for every compact subset $K$ of $E$. Equivalently, $E$ is compactly generated if and only if each function $f\colon E \to \Omega$ to some topological space $\Omega$ is continuous if and only if it is continuous on every compact subset of $E$. Note that every locally compact space and every metric space (or, more generally, every topological space that is first countable) is compactly generated.

\begin{lem} \label{l.cbm}
    Let $E$ be a completely regular Hausdorff topological space. Then the following assertions hold:
    \begin{enumerate}[\upshape (a)]
        \item \label{i.l.cbm.normingPair}
              $(C_b(E), \calM(E))$ is a norming dual pair with respect to the duality~\eqref{eq.duality}.

        \item \label{i.l.cbm.convergenceWeakTop}
              A sequence $(f_n)_n$ in $C_b(E)$ is $\sigma(C_b(E), \calM(E))$-convergent to a function $f \in C_b(E)$ 
              if and only if it is uniformly bounded and converges pointwise to $f$.

        \item \label{i.l.cbm.DunfordPettisProp}
              $C_b(E)$  has the Dunford--Pettis property with respect to $\calM(E)$.

        \item \label{i.l.cbm.ConsistencyStrictTop}
              The strict topology is a consistent topology on $C_b(E)$, i.e., $(C_b(E), \beta_0(E))' = \calM(E)$. In particular, it is coarser than the Mackey topology. Moreover, $\beta_0(E)$ coincides on $\|\argument\|_\infty$-bounded sets with the compact-open topology of uniform convergence on compact sets.

        \item \label{i.l.cbm.CompletenessStrictTop}
              The space $(C_b(E), \beta_0(E))$ is complete if and only if $E$ is compactly generated.
    \end{enumerate}
\end{lem}
\begin{proof}
    \ref{i.l.cbm.normingPair}:
    This is well known (see, e.g.,~\cite[Example~2.4]{Kunze2011}). \smallskip

    \ref{i.l.cbm.convergenceWeakTop}:
    If the sequence $(f_n)_n$ is uniformly bounded and converges pointwise to $f$, then it follows from the dominated convergence theorem that $f_n \to f$ with respect to $\sigma(C_b(E), \calM(E))$ as $n \to \infty$. For the converse, use Proposition \ref{p.weaklycontinuous}.\smallskip

    \ref{i.l.cbm.DunfordPettisProp}:
    Taking~\ref{i.l.cbm.convergenceWeakTop} into account, we see that the lattice operations on $C_b(E)$ are sequentially $\sigma(C_b(E),\calM(E))$-continuous. Since $\calM(E)$ is an AL-space, the claim follows from Proposition~\ref{p.ALDP}. \smallskip

    \ref{i.l.cbm.ConsistencyStrictTop}:
    See, for example,~\cite[Theorem~7.6.3]{Jarchow1981} for the fact that the strict topology is consistent on $C_b(E)$. Thus, it follows from the Mackey--Arens Theorem (see for example \cite[Theorem~5.113]{Aliprantis2006}) that $\beta_0(E)$ is coarser than the Mackey topology $\mu(C_b(E), \calM(E))$. Finally, the strict topology coincides on $\|\argument\|_\infty$-bounded sets with the compact-open topology of uniform convergence on compact sets by~\cite[Theorem~2.10.4]{Jarchow1981}. \smallskip

    \ref{i.l.cbm.CompletenessStrictTop}:
    This was proven independently by Sentilles~\cite[Theorem~7.1]{Sentilles1972} and Hoffmann-J{\o}rgensen~\cite[Theorem~7.1]{HoffmannJoergensen1972}.
\end{proof}

In the concrete situation of the norming dual pair $(C_b(E), \calM(E))$, Lemma~\ref{l.eberlein_smulian} together with Corollary~\ref{c.ultrafeller} yields the following result.

\begin{cor}\label{c.situationoncb}
    Let $E$ be a completely regular Hausdorff topological space and $S, T \in \calL(C_b(E))$ be $\sigma(C_b(E), \calM(E))$-compact. Then the following assertions hold:
    \begin{enumerate}[\upshape (a)]
        \item \label{i.c.situationcb.sequentialConvImprov}
              $T$ is sequentially $\sigma(C_b(E), \calM(E))$-$\beta_0(E)$-continuous, i.e., if $(f_n)_n$ is a uniformly bounded sequence in $C_b(E)$ that converges pointwise to some $f \in C_b(E)$, then $Tf_n \to Tf$ as $n \to \infty$ uniformly on compact subsets of $E$.

        \item \label{i.c.situationcb.abstractProductCompact}
              Assume, additionally, that $\calM(E)$ contains a sequence that separates points in $C_b(E)$. Then the product $TS$ is sequentially $\beta_0(E)$-compact, i.e., whenever $(f_n)_n$ is a bounded sequence in  $C_b(E)$, we may extract a subsequence $(f_{n_k})_k$ such that $ST f_{n_k}$ converges with respect to $\beta_0(E)$ to a continuous function as $k \to \infty$.

        \item \label{i.c.situationcb.cgsProductCompact}
              If $E$ is compactly generated and $\calM(E)$ contains a sequence that separates points in $C_b(E)$, then $ST$ is $\beta_0(E)$-compact and sequentially $\beta_0(E)$-compact.

        \item \label{i.c.situationcb.csProductCompact}
              If $E$ is compact, then $TS$ is $\|\argument\|_\infty$-compact. 
    \end{enumerate}
\end{cor}
\begin{proof}
    \ref{i.c.situationcb.sequentialConvImprov}:
    Since by Lemma~\ref{l.cbm}  $C_b(E)$ has the Dunford--Pettis property with respect to $\calM(E)$, Theorem~\ref{t.dpimproving} implies that $T$ is sequentially $\sigma(C_b(E), \calM(E))$-$\mu(C_b(E), \calM(E))$-continuous. As the strict topology $\beta_0(E)$ is coarser than  the Mackey topology $\mu(C_b(E), \calM(E))$, \ref{i.c.situationcb.sequentialConvImprov} follows. \smallskip

    \ref{i.c.situationcb.abstractProductCompact}:
    By Lemma~\ref{l.eberlein_smulian}\ref{i.l.eberlein_smulian.smulian} $S$ is sequentially $\sigma(C_b(E), \calM(E))$-compact, and Corollary~\ref{c.ultrafeller}\ref{i.c.ultrafeller.TSSequentiallyMackeyCompact} implies  sequential $\mu(C_b(E), \calM(E))$-compactness of $TS$. As $\beta_0(E) \subset \mu(C_b(E), \calM(E))$, the statement follows. \smallskip

    \ref{i.c.situationcb.cgsProductCompact}:
    This follows as in~\ref{i.c.situationcb.abstractProductCompact} by using part \ref{i.c.ultrafeller.TSMackeyCompact} of Corollary~\ref{c.ultrafeller} instead of~\ref{i.c.ultrafeller.TSSequentiallyMackeyCompact}. \smallskip

    \ref{i.c.situationcb.csProductCompact}:
    Since $E$ is compact, we have $C_b(E)^* \cong \calM(E)$. This implies that the $\sigma(C_b(E), \calM(E))$-topology coincides with 
    $\sigma(C_b(E), C_b(E)^*)$ and Proposition \ref{p.es} implies that the notions of relative weak sequential compactness and relative weak
    compactness of sets coincide. Hence, $S$ is sequentially $\sigma(C_b(E), \calM(E))$-compact and 
    Corollary~\ref{c.ultrafeller}\ref{i.c.ultrafeller.TSSequentiallyMackeyCompact} implies the sequential $\mu(C_b(E), \calM(E))$-compactness of 
    $TS$. Noting that $\beta_0(E)$, $\mu(C_b(E), \calM(E))$ and the norm topology coincide, the statement follows.
\end{proof}

The existence of a sequence in $\calM(E)$ that separates the points
in $C_b(E)$, which is assumed in parts (b) and (c) of Corollary
\ref{c.situationoncb}, is certainly fulfilled whenever $E$ is separable. In Appendix \ref{apx:separability}, we discuss if the separability of $E$ is also necessary for this.

The next example shows that parts~\ref{i.c.situationcb.abstractProductCompact} and~\ref{i.c.situationcb.cgsProductCompact} of Corollary~\ref{c.situationoncb} are not true in general without the separability assumption.

\begin{example}
\label{ex.nonseq}
    Let $E \coloneqq \R$, endowed with the discrete metric. Then $E$ is compactly generated and the compact subsets of $E$ are precisely the finite subsets. In this case, we may identify $C_b(E) \cong \ell^\infty(E)$ and $\calM(E)\cong \ell^1(E)$. We note that $\sigma(C_b(E), \calM(E))$ is merely the weak$^*$-topology (in the Banach space sense) on $\ell^\infty(E)$. It follows from the Banach--Alaoglu theorem that the identity map $\id \colon C_b(E)\to C_b(E)$ is $\sigma(C_b(E), \calM(E))$-compact. As the compact subsets of $E$ are all finite, the $\sigma(C_b(E), \calM(E))$-topology coincides with the strict topology $\beta_0(E)$, and thus, assertion~\ref{i.c.situationcb.sequentialConvImprov} in Corollary~\ref{c.situationoncb} is trivially satisfied. However, $\calM(E)\cong \ell^1(E)$ does not contain a sequence that separates the points in $C_b(E)$. And, indeed, in this situation $\id^2 = \id$ is \emph{not} sequentially $\beta_0(E)$-compact (even though it is $\beta_0(E)$-compact). 
\end{example}

Next, we show under appropriate assumptions on the space $E$ that the strong Feller property and the ultra Feller property of a kernel operator are equivalent to the $\sigma(C_b(E), \calM(E))$-compactness and the $\beta_0(E)$-compactness, respectively. Let us first recall the notion of kernels.

\begin{defn}
A \emph{kernel} on $E$ is a map $k \colon E \times \calB(E)\to \K$ with the following properties:
\begin{enumerate}[(i)]
    \item
          for every $A\in \calB(E)$, the map $x\mapsto k(x,A)$ is $\calB(E)$-measurable.

    \item
          for every $x\in E$, it is $k(x, \argument) \in \calM(E)$.

    \item
          $\|k\|\coloneqq \sup_{x\in E}|k|(x, E) <\infty$.
\end{enumerate}
In what follows, a kernel $k$ on $E$ is said to be
\begin{enumerate}[(a)]
    \item a \emph{strong Feller kernel} if it is \emph{setwise continuous}, meaning that for every $A \in \calB(E)$, the scalar-valued map $x \mapsto k(x, A)$ is continuous on $E$. Note
    that this is equivalent to the $\sigma(\calM(E), B_b(E))$-continuity of the map $x\mapsto k(x, \cdot)$.
    \item It is called an \emph{ultra Feller kernel} if the map
              \begin{align*}
                  E \to \calM(E), \quad x \mapsto k(x, \argument)
              \end{align*}
    is continuous with respect to the total variation norm on $\calM(E)$.
\end{enumerate}
\end{defn}

Evidently, each ultra Feller kernel is a strong Feller kernel. The following lemma shows that one can associate a unique bounded operator on $B_b(E)$ to each kernel on $E$. The proof is straightforward and will thus be omitted. 

\begin{lem} \label{l.kerneloperator}
    Let $k$ be a kernel on $E$. Then there exists a bounded operator $T_k \colon B_b(E) \to B_b(E)$ such that 
    $(T_k\one_A)(x) = k(x , A)$ for all $x\in E$ and $A \in \calB(E)$. Moreover, the identity
    \begin{align*}
        [T_k f](x) = \int_E f(y) \, k(x, \ud y)
    \end{align*}
    holds for all  $x \in E$ and $f \in B_b(E)$. In this situation, $k$ is a strong Feller kernel if and only if one has $T_kf \in C_b(E)$ for all $f \in B_b(E)$. 
\end{lem}

We call the operator from the lemma above, the \emph{kernel operator} associated to the kernel $k$. Next, we consider the restrictions of kernel operators to closed subspaces of the bounded measurable functions. Let $X$ be a closed linear subspace of $B_b(E)$ that is norming for $\calM(E)$, e.g., $X = C_b(E)$ or, if $E$ is locally compact, $X = C_0(E)$. A bounded operator $T$ on $X$ is called a \emph{kernel operator} if there exists a kernel $k$ such that $Tf=T_kf$ for all $f\in X$. Note that, by Lemma~\ref{l.kerneloperator} and the fact that $X$ separates
the points in $\calM(E)$, 
$k$ is uniquely determined by $T$ and will be called \emph{the kernel associated with $T$}. 
Conversely, $T$ is called the \emph{kernel operator associated with $k$}. We also point out that the fact that $T$ maps functions $f \in X$ to $X$ generally does not imply that $k(\argument, A) \in X$ for all $A\in \calB(E)$. 

\begin{defn}
    Let $T \colon C_b(E) \to C_b(E)$ be a kernel operator with kernel $k$. We consider the following two notions: 
    \begin{enumerate}[(a)]
        \item
              $T$ is called a \emph{strong Feller operator} if its kernel $k$ is a strong Feller kernel. 

        \item
              $T$ is called an \emph{ultra Feller operator} if its kernel $k$ is an ultra Feller kernel. 
    \end{enumerate}
\end{defn}

\begin{rem}
    We point out that while the notions of strong and ultra Feller
    operators introduced above are rather standard, the notion of a
    \emph{Feller operator} is not commonly used. There is, however,
    the notion of a Feller semigroup, even though its definition
    is inconsistent in the literature (see the comments 
    in \cite[p.~241]{Rogers2000}). 
    The definition from \cite{Jacob2001} is as follows:
    If $E$ is a locally compact metric space, then \emph{a Feller semigroup} is a strongly continuous semigroup $(T(t))_{t\geq 0}$
    of positive contractions on the space $C_0(E)$ of continuous functions on $E$ that vanish at infinity. 
    In this context, the strong and ultra Feller properties arise
    as additional properties that the operators $T(t)$ may (or may not) have. However, it can happen
    that a kernel operator (and even a semigroup of kernel operators)
    satisfies the strong or ultra Feller property without leaving
    the space $C_0(E)$ invariant, see \cite{Bogdan2024} for an example. 
\end{rem}

Evidently, each ultra Feller kernel operator has the strong Feller property. As shown in the following theorem, for a large class of spaces $E$, $\sigma(C_b(E), \calM(E))$-compact operators are precisely strong Feller operators. This result generalizes \cite[Theorem~1(2)]{Sentilles1969}, where this characterization was established in the special case of a locally compact space $E$. 

\begin{thm}\label{t.strongfellerop}
    Let $E$ be a completely regular Hausdorff topological space that is compactly generated and $T\in \calL(C_b(E))$. Then the following assertions are equivalent:
    \begin{enumerate}[\upshape (i)]
        \item \label{i.t.strongfellerop.weakCompact}
              $T$ is $\sigma (C_b(E), \calM(E))$-compact.
        \item \label{i.t.strongfellerop.strongFeller}
              $T$ is a strong Feller operator.
    \end{enumerate}
\end{thm}
\begin{proof}
    \impliesProof{i.t.strongfellerop.weakCompact}{i.t.strongfellerop.strongFeller}
    By definition, an $\sigma(C_b(E), \calM(E))$-compact operator $T$ is also $\sigma(C_b(E), \calM(E))$-continuous, whence we may consider 
    the $\sigma(C_b(E), \calM(E))$-adjoint $S=T'$ of $T$ on $\calM(E)$. We put $k(x, \argument) \coloneqq T'\delta_x \in \calM(E)$. 
    Since $C_b(E)$ is norming for $\calM(E)$, it follows that
    \[
        |k|(x,E) =\sup\{ |Tf(x)| : f \in C_b(E), \|f\|_\infty \leq 1 \}.
    \]
    Taking the supremum over all $x\in E$, it follows that $\|k\| =\|T\| <\infty$. Finally, note that
    \[
        k(x,A) = \langle \one_A, T'\delta_x\rangle_{\calM^*, \calM} =
        \langle S^*\one_A, \delta_x\rangle_{\calM^*, \calM} = [S^*\one_A](x).
    \]
    Since $T$ is $\sigma(C_b(E), \calM(E))$-compact, it follows from Proposition~\ref{p.weakcomp} that $k(\argument, A) = S^*\one_A \in C_b(E)$. Thus, $T = S^*\rvert_{B_b(E)}$ is a strong Feller operator by Lemma~\ref{l.kerneloperator}. \medskip
    
   \impliesProof{i.t.strongfellerop.strongFeller}{i.t.strongfellerop.weakCompact}
    Let $T$ be a kernel operator with a setwise continuous kernel. Given $\mu \in \calM(E)$, we define
    \[
        [S\mu](A) = \int_E k(x, A) \, \ud \mu (x),
    \]
    where $k$ denotes the kernel associated to $T$. Making use of the dominated convergence theorem, it is easy to see that $S\mu$ is a Borel measure. Let us prove that it is actually a Radon measure. To that end, let a set $A\in \calB(E)$ and $\eps>0$ be given. We first pick $L \subset E$ compact such that $|\mu|(E\setminus L) \leq \eps/(2\|k\|)$. Since $k$ is setwise continuous, the map $x\mapsto k(x, \argument)$ is $\sigma(\calM(E), B_b(E))$-continuous, whence the set
    \begin{equation}\label{eq.sets}
        \calS \coloneqq \{k(x, \argument) : x\in L\}
    \end{equation}
    is $\sigma(\calM(E), B_b(E))$-compact. It follows from \cite[Theorem~4.7.25]{Bogachev2007} that there exists a non-negative Radon measure $\nu$, such that $\calS$ is uniformly $\nu$-additive. Thus, we find $\delta>0$ such that for all $x\in L$ and $B\in\calB(E)$ with $\nu(B)\leq \delta$
    we have $|k|(x, B)\leq \eps/(2|\mu|(E))$.

    Now pick $K\subset A$ compact with $\nu (A\setminus K) \leq \delta$. It follows that
    \[
        |S\mu|(A\setminus K) \leq \int_L |k|(x, A\setminus K)\, \ud|\mu|(x) + \int_{E\setminus L} \|k\| \, \ud|\mu|(x)
        \leq \eps,
    \]
    which proves that $S\mu$ is indeed a Radon measure. It is clear from the definition that $\langle Tf, \mu\rangle = \langle f, S\mu\rangle$, whence $T^*\calM(E) \subset \calM(E)$ and $T^*|_{\calM(E)} = S$. By Proposition~\ref{p.weaklycontinuous}, $T$ is $\sigma (C_b(E), \calM(E))$-continuous and $T' = S$.

    Taking Proposition~\ref{p.weakcomp} into account, it remains to prove that $S^*\calM(E)^*\subset C_b(E)$. To that end, fix $\varphi\in \calM(E)^*$ and put $f(x) = \langle S^*\varphi, \delta_x\rangle$. We thus have to prove that $f$ is a continuous function. As $E$ is compactly generated, it suffices to prove that $f$ is continuous on any compact subset of $E$. So fix a compact set $L$ and again consider the set $\calS$ from~\eqref{eq.sets}, together with the measure $\nu$, obtained from~\cite[Theorem~4.7.25]{Bogachev2007}.

    By the Radon--Nikodym Theorem, for any $x \in L$, we find a function $g_x\in L^1(\nu)$ with $k(x, \ud y) = g_x\nu(\ud y)$. By Proposition~\ref{p.dense}, there is a bounded net $(f_\lambda)_\lambda$ that converges to $\varphi$ with respect to $\sigma(\calM(E)^*, \calM(E))$. Identifying a bounded measurable function with its equivalence class modulo equality $\nu$-almost everywhere, we may consider $(f_\lambda)_\lambda$ as a bounded net in $L^\infty(\nu)$. By the Banach--Alaoglu theorem, passing to a subnet, we may and shall assume that $(f_\lambda)_\lambda$ converges to some element $[g] \in L^\infty(E, \nu)$. We pick any bounded representative $g$. Then, for every $x\in L$, we have
    \begin{align*}
        f(x) = \langle S^*\varphi, \delta_x\rangle_{\calM^*, \calM} & = \langle \varphi, k(x, \argument)\rangle_{\calM^*, \calM} = \lim_\lambda \langle f_\lambda, k(x, \argument)\rangle_{B_b,\calM} \\
        & = \lim_\lambda \langle f_\lambda, g_x\rangle_{L^\infty(\nu), L^1(\nu)}
        = \langle g, g_x\rangle_{L^\infty(\nu), L^1(\nu)} \\
        & = \langle g, S\delta_x\rangle_{B_b, \calM} =(Tg)(x).
    \end{align*}
    As $k$ is setwise continuous, $Tg$ is a continuous function. This proves that $f$ is indeed continuous on $L$.
\end{proof}

\begin{rem}
    Making use of Theorem~\ref{t.strongfellerop}, we can improve \cite[Proposition~A.2]{Gerlach2023}: Let $E$ be a completely regular Hausdorff topological space that is compactly generated, let $T$ be a strong Feller operator on $C_b(E)$ and let $S \coloneqq T'$. Then $T$ is $\sigma(C_b(E), \calM(E))$-compact by Theorem~\ref{t.strongfellerop}, and it follows from Proposition~\ref{p.weakcomp} that $S^* \calM(E)^* \subset C_b(E)$, i.e, the conclusion of \cite[Proposition~A.2]{Gerlach2023} holds.  
    
    However, \cite[Proposition~A.2]{Gerlach2023} requires stronger assumptions, namely that $E$ is a Polish space and that $T$ is ultra Feller. 
\end{rem}

We now connect the ultra Feller property of a kernel operator to its compactness in the strict topology $\beta_0(E)$. The main ingredient in the proof of the following theorem is the Arzelà--Ascoli theorem. 

\begin{thm}\label{t.ultrafellerop}
    Let $E$ be a completely regular Hausdorff topological space that is compactly generated, and let $T \in \calL(C_b(E))$ be a bounded linear operator. Then the following assertions are equivalent:
    \begin{enumerate}[\upshape (i)]
        \item \label{i.c.ultrafellerop.strictCompact}
            $T$ is $\beta_0(E)$-compact;
        \item \label{i.c.ultrafellerop.ultraFeller}
            $T$ is an ultra Feller operator.
    \end{enumerate}
\end{thm}

\begin{proof}
    \impliesProof{i.c.ultrafellerop.strictCompact}{i.c.ultrafellerop.ultraFeller}
    If $T$ is $\beta_0(E)$-compact, then $T$ is $\sigma(C_b(E), \calM(E))$-compact and, therefore, a strong Feller operator by
    Theorem~\ref{t.strongfellerop}. Hence, it is only left to show that the kernel $k$ of $T$ is continuous with respect to the total variation norm.
    Since $E$ is compactly generated, it is sufficient to show that $k$ is continuous on every compact subset of $E$. 
    So let $K \subset E$ be compact, $x \in K$ and $\eps > 0$. 
    Since $T$ is $\beta_0(E)$-compact, $T \Ball_{C_b(E)}$ is a relatively $\beta_0(E)$-compact subset of $C_b(E)$.
    Restricting to the set $K$, we see that $T \Ball_{C_b(E)}\rvert_K$ is a relatively compact subset of $C(K)$ and the Arzelà--Ascoli theorem 
    implies that $T \Ball_{C_b(E)}$ is equicontinuous on $K$. 
    Therefore, given $\eps > 0$, there exists a (relatively) open neighbourhood $U_x \subset K$ of $x$ such 
    that $|Tf(x) - Tf(y)| \leq \eps$ for all $f \in \Ball_{C_b(E)}$ and $y\in U_x$. Consequently,
    \begin{align*}
        \norm{k(x, \argument) - k(y, \argument)}_{\mathrm{TV}} & =  \sup\{ |Tf(x) - Tf(y)| : f \in \Ball_{C_b(E)}\} \leq \eps
    \end{align*}
    for all $y \in U_x$. This shows that $T$ is an ultra Feller operator. \smallskip
    
    \impliesProof{i.c.ultrafellerop.ultraFeller}{i.c.ultrafellerop.strictCompact}
    We show that the set $T \Ball_{C_b(E)}$ is relatively $\beta_0(E)$-compact. First, $T \Ball_{C_b(E)}$ is obviously norm bounded. As the strict topology coincides with the compact-open topology on norm bounded sets, $T\Ball_{C_b(E)}$ is relatively $\beta_0(E)$-compact if and only if it is relatively compact in the compact-open topology. The latter is equivalent to $T\Ball_{C_b(E)}$ being equicontinuous on every compact subset of $E$ by the Arzelà--Ascoli theorem. To prove equicontinuity, let $\varepsilon > 0$ and fix $x \in E$. As $T$ is ultra Feller, there exists an open neighbourhood $U_x \subset E$ of $x$ such that $\norm{k(x, \argument) - k(y, \argument)}_{\mathrm{TV}} \leq \eps$ for all $y \in U_x$. Thus,
    \begin{align*}
        \modulus{Tf(x) - Tf(y)} 
        & = \modulus[\bigg]{\int_E f(z) \, k(x, \mathrm d z) - \int_E f(z) \, k(y, \mathrm d z)} \\
        &\leq \norm{k(x, \argument) - k(y, \argument)}_{\mathrm{TV}} \leq \eps
    \end{align*}
    for all $f \in \Ball_{C_b(E)}$ and $y \in U_x$. This proves equicontinuity of $T\Ball_{C_b(E)}$.
\end{proof}

We can now prove the following result on the connection between strong Feller and ultra Feller operators.

\begin{thm}\label{t.productstrongfeller}
    Let $E$ be a compactly generated, completely regular Hausdorff topological space and let $S, T \colon C_b(E) \to C_b(E)$ be kernel operators with the strong Feller property. Suppose that at least one of the following conditions is satisfied.
    \begin{enumerate}[\upshape (a)]
        \item 
            There exists a sequence of measures in $\calM(E)$ that separate the points in $C_b(E)$.
            
        \item 
            $E$ is compact.
    \end{enumerate}
    Then $TS$ is a kernel operator with the ultra Feller property.
\end{thm}

\begin{proof}
    In both cases, $T$ and $S$ are $\sigma(C_b(E), \calM(E))$-compact operators by Theorem~\ref{t.strongfellerop}. If $E$ is compact, 
    then $T$ and $S$ are also sequentially $\sigma(C_b(E), \calM(E))$-compact by Proposition \ref{p.es}, as in this case $\calM(E) \cong C_b(E)^*$.
    But also if $\calM(E)$ contains a sequence that separates the points in $C_b(E)$, $T$ and $S$ are sequentially $\sigma(C_b(E), \calM(E))$-compact
    by Lemma~\ref{l.eberlein_smulian}\ref{i.l.eberlein_smulian.smulian}. By Corollary~\ref{c.situationoncb}, 
    the product $TS$ is $\beta_0(E)$-compact and Theorem~\ref{t.ultrafellerop} shows that $TS$ is an ultra Feller operator.
\end{proof}

We end this section by giving an example of a strong Feller operator that does not satisfy the ultra Feller property. This shows that
the notions of `strong Feller operator' and `ultra Feller operator' are distinct, even if
$E$ is a metrizable and compact set.

\begin{example}
    On $E = [0, 1]$, we consider a kernel operator with kernel $k$, where every measure $k(x, \argument)$ has a density with respect to Lebesgue measure. More precisely, if $m\colon [0,1] \to L^1(0,1)$, we write $m(x,y)$ shorthand for $[m(x)](y)$ and define $k$ by setting
    \begin{align*}
        k(x, A) \coloneqq \int_A m(x,y)\, \ud y \qquad \text{for all } A \in \calB([0,1]).
    \end{align*}
    To obtain a kernel which satisfies the strong Feller property but not the ultra Feller property, we choose $m$ as follows.

    Let ${(r_n)}_n$ be the sequence of Rademacher functions, i.e.,
    \begin{align*}
        r_n \colon [0,1] \to \R, \quad y \mapsto \operatorname{sgn}(\sin(2^n \pi y)).
    \end{align*}
    It is well known that the sequence of Rademacher functions converges weakly to 0 in $L^1(0,1)$. Adding the constant $\one$ function, we set $m(\frac{1}{n}, y) \coloneqq 1+ r_n(y)$ for $n\in \N$ and $y\in (0,1)$ and $m(0, y) \equiv 1$. 
    We extend this function to $[0,1]\times (0,1)$ by linear interpolation. Thus, if $x= \frac{\lambda}{n} + \frac{1-\lambda}{n+1}$ for some $n\in \N$ and $\lambda \in (0,1)$, we set $m(x, y) = \lambda m(\frac{1}{n}, y) + (1-\lambda)m(\frac{1}{n+1}, y)$. 
    Note that for $A\in \calB([0,1])$ it is $k(x, A) = \langle m(x), \one_A\rangle_{L^1, L^\infty}$. Thus, the weak$^*$-convergence of $m(\frac{1}{n})$ to $m(0)$ as $n\to \infty$ yields the continuity of $x\mapsto k(x,A)$. This shows that $k$ enjoys the strong Feller property.
    
    On the other hand, $k$ does not have the ultra Feller property, as 
    \begin{align*}
        \norm{k(\tfrac{1}{n}, \argument) - k(0, \argument)}_{\mathrm{TV}} = \norm{m(\tfrac{1}{n}) - m(0)}_{L^1} = \norm{r_n}_{L^1} \equiv 1
        \not\to 0.
    \end{align*}
	Thus $x\mapsto k(x, \argument)$ is not continuous with respect to the total variation norm.
\end{example}

\section{Further applications} \label{sect.applications}

\subsection{Application to the pair \texorpdfstring{$(C_b(E), L^1(\fm))$}{(Cb(E), L1(m))}}
\label{ssect.a1}

In many applications, the state space $E$ is endowed with a certain `natural measure' $\fm$ and one wants to work with the space $L^1(\fm)$ of measures absolutely continuous with respect to $\fm$ rather than with $\calM(E)$, the space of \emph{all} Radon measures. The most prominent example is $E = \R^d$ (or open subsets thereof) endowed with Lebesgue measure. In what follows, we will consider $L^1(\fm)$
as a closed subspace of $\calM(E)$ via the isometric embedding
\[
    L^1(\fm) \to \calM(E), \quad g \mapsto g\, \ud \fm.
\]
Throughout, we assume the following:

\begin{hypo}\label{h.1}
Let $E$ be compactly generated and $\fm$ be a $\sigma$-finite, non-negative Radon measure on the Borel $\sigma$-algebra $\calB(E)$
with the following property. For every $x\in E$ there is a sequence $(g_{n, x})_n$ in $L^1(\fm)$ with $\|g_{n, x}\|_1 \leq 1$ for all $n$
and $g_{n, x} \to \delta_x$ with respect to $\sigma(\calM(E), C_b(E))$.
\end{hypo}

\begin{rem}
    Hypothesis \ref{h.1} is satisfied whenever $E$ is a metric space and the measure $\fm$ satisfies $0<\fm (B(x,\eps))<\infty$ for each $x \in E$ and $\eps > 0$, where $B(x, \eps)$ denotes the closed unit ball of radius $\eps$ centred around $x$. In this case,
    the sequence $(g_{n,x})_n$, given by
    \[
        g_{n,x} \coloneqq \frac{1}{\fm(B(x,n^{-1}))}\one_{B(x, n^{-1})},
    \]
    is in $L^1(\fm)$ and satisfies $g_{n,x} \to \delta_x$ with respect to $\sigma(\calM(E), C_b(E))$.
\end{rem}

\begin{lem}\label{l.cbl1dual}
    Assume Hypothesis \ref{h.1}. Then the following assertions hold: 
    \begin{enumerate}[\upshape (a)]
        \item \label{i.l.cbl1dual.normingPair}
              $(C_b(E), L^1(\fm))$ is a norming dual pair.

        \item \label{i.l.cbl1dual.dunfordPettisProp}
              $C_b(E)$ has the Dunford--Pettis property with respect to $L^1(\fm)$.
    \end{enumerate}
\end{lem}

\begin{proof}
    \ref{i.l.cbl1dual.normingPair}:
    It follows from the assumption that for every $x\in E$, it is
    \[
        |f(x)| = \lim_{n\to\infty} |\langle f, g_{n,x}\rangle| \leq \sup\{|\langle f, g\rangle | : g\in L^1(\fm), \|g\|_1
        \leq 1\}
    \]
    This implies that $L^1(\fm)$ is norming for $C_b(E)$. \smallskip

    \ref{i.l.cbl1dual.dunfordPettisProp}:
    Since $L^1(\fm)$ is an AL-space and the lattice operations on $C_b(E)$ are sequentially $\sigma (C_b(E), L^1(\fm))$-continuous, this follows from Proposition~\ref{p.ALDP}.
\end{proof}

We can now  deduce the following characterization of $\sigma(C_b(E), L^1(\fm))$-compact operators on the norming dual pair $(C_b(E), L^1(\fm))$. 

\begin{thm}\label{t.cbl1sf}
    Assume that Hypothesis~\ref{h.1} holds and let $T\in \calL(C_b(E))$. Then the following assertions are equivalent:
    \begin{enumerate}[\upshape (i)]
        \item \label{i.t.cbl1sf.weakCompact}
              $T$ is $\sigma(C_b(E), L^1(\fm))$-compact.

        \item \label{i.t.cbl1sf.weakMECompactAndInvariant}
              $T$ is $\sigma(C_b(E), \calM(E))$-compact and $T'\calM(E)\subset L^1(\fm)$.

        \item \label{i.t.cbl1sf.kernelSetwiseCont}
              $T$ is a kernel operator whose kernel $k$ is setwise continuous and $k(x, \argument) \in L^1(\fm)$ for every $x\in E$.
    \end{enumerate}
\end{thm}

\begin{proof}
    \impliesProof{i.t.cbl1sf.weakCompact}{i.t.cbl1sf.kernelSetwiseCont}
    We first note that $L^1(\fm)^*\cong L^\infty(\fm)$. It follows from Proposition~\ref{p.weakcomp}, that for $S = T'\in \calL(L^1(\fm), \sigma(L^1(\fm), C_b(E)))$, we have $S^*L^\infty(\fm) \subset C_b(E)$. We denote by $q\colon B_b(E)\to L^\infty(E)$ the quotient map that maps a bounded, measurable function to its equivalence class modulo equality $\fm$-almost everywhere. Setting $R=S^*\circ q$, it follows that $R\in \calL(B_b(E))$ and $R B_b(E) \subset C_b(E)$. Moreover, if $(f_n)_n$ is a bounded sequence in $B_b(E)$ that converges pointwise to $f$, then $q(f_n) \to q(f)$ with respect to $\sigma(L^\infty(\fm), L^1(\fm))$. As $S^*$ is an adjoint operator, $Rf_n \to Rf$ with respect to $\sigma (L^\infty(\fm), L^1(\fm))$. Now fix $x\in E$ and put $k_n(x,A) \coloneqq \langle R\one_A, g_{n,x}\rangle$. It follows from the above continuity property that $A \mapsto k_n(x,A)$ is a measure on $E$. Since $R\one_A$ is a continuous function, $k(x, A) = \lim_{n \to \infty} k_n(x, A) = [R\one_A](x)$ exists for every $A\in \calB(E)$. By the Vitali--Hahn--Saks Theorem (see \cite[Theorem~4.6.3]{Bogachev2007}) $k(x, \argument)$ is a measure that is absolutely continuous with respect to $\fm$.

    By construction, $k$ is a setwise continuous kernel and $T$ is associated with $k$. Thus, \ref{i.t.cbl1sf.kernelSetwiseCont} is proved. \smallskip

    \impliesProof{i.t.cbl1sf.kernelSetwiseCont}{i.t.cbl1sf.weakMECompactAndInvariant}
    This follows from Theorem~\ref{t.strongfellerop} and the observation that if the kernel $k$ satisfies $k(x, \argument)\in L^1(\fm)$, then for any $A\in \calB(E)$ with $\fm(A)=0$
    \[
        (T'\mu)(A) = \int_E k(x, A)\, \ud \mu (x) = 0,
    \]
    so that $T'\mu\in L^1(\fm)$. \smallskip

    \impliesProof{i.t.cbl1sf.weakMECompactAndInvariant}{i.t.cbl1sf.weakCompact}
    It is immediate from (ii) that $T$ is $\sigma(C_b(E), L^1(\fm))$-continuous. As $\sigma(C_b(E), \calM(E))$ is stronger than $\sigma(C_b(E), L^1(\fm))$, (i) follows.\smallskip

    \impliesProof{i.t.cbl1sf.weakCompact}{i.t.cbl1sf.weakMECompactAndInvariant}
    This is obvious. 
\end{proof}

\begin{rem}
    We note that the assumption that $E$ is compactly generated is only needed for the characterization of strong Feller operators in Theorem \ref{t.strongfellerop} which is used in the proof of (iii) $\Rightarrow$ (ii). The implications (ii) $\Rightarrow$ (i) $\Rightarrow$ (iii) remain valid without this assumption.
\end{rem}

\begin{rem}
    We point out that a sequence $(f_n)_n$ in $C_b(E)$ that converges with respect to $\sigma(C_b(E), L^1(\fm))$ to a function $f\in C_b(E)$
    need not converge pointwise (and not even pointwise almost everywhere). On the other hand, 
    convergence with respect to $\sigma(C_b(E), \calM(E))$ implies pointwise convergence. In view of these facts, the equivalence 
    of $\sigma(C_b(E), L^1(\fm))$-compactness and $\sigma(C_b(E), \calM(E))$-compactness in Theorem \ref{t.cbl1sf} is remarkable.
\end{rem}

\subsection{Application to the pair \texorpdfstring{$(C_0(E), \calM(E))$}{(C0(E), M(E))}}
\label{ssect.a2}

In this subsection, $E$ is assumed to be locally compact and $C_0(E)$ denotes the space of continuous functions $f$ on $E$ that \emph{vanish at infinity}, i.e., that for each $\eps > 0$, there exists a compact set $K \subset E$ such that $|f(x)| \leq \eps$ for all $x\in E \setminus K$. Endowed with the supremum norm, $C_0(E)$ is a Banach space and its dual space is isometrically isomorphic to $\calM(E)$.

On the norming dual pair $(C_0(E), \calM(E))$, the topology $\sigma(C_0(E),\calM(E))$ is the weak topology in the Banach space sense. 
On the other hand, the Mackey topology $\mu(C_0(E), \calM(E))$ coincides with the norm topology. Thus, a $\sigma(C_0(E),\calM(E))$-compact operator on $C_0(E)$ is nothing but a `weakly compact operator' and a $\mu(C_0(E), \calM(E))$-compact operator is just a `compact operator'. We should also point out that $C_0(E)$ has the classical Dunford--Pettis property, i.e., $C_0(E)$ has the Dunford--Pettis property with respect to $\calM(E)$. In particular, the following classical result of Dunford and Pettis holds (see, e.g.,~\cite[Corollary~5.88]{Aliprantis2006a}). 

\begin{cor}
    Let $E$ be locally compact and $T, S$ be weakly compact operators on $C_0(E)$. Then $TS$ is compact.
\end{cor}

With similar arguments as in Theorem~\ref{t.strongfellerop} and Theorem~\ref{t.ultrafellerop}, 
we can give a characterization of (weakly) compact operators on $C_0(E)$. As a preparation, we state the following lemma. The proof is straightforward. 

\begin{lem} \label{l.kernel.compactification}
    Let $E$ be locally compact and let $\hat E = E \cup \set{\infty}$ be its one-point compactification. Let $k$ be a kernel on $E$. Then the map
    \begin{align*}
        \hat k \colon \hat E \times \calB(\hat E) \to \bbK, \quad k(x, B) \coloneqq
        \begin{cases}
            k(x, B \cap E), \quad &\text{if } x \in E, \\
            0, \quad &\text{if } x = \infty,
        \end{cases}
    \end{align*}
    is a kernel on $\hat E$. If, additionally, one has $k(\argument, A) \in C_0(E)$ for all $A \in \calB(E)$, then $\hat k(\argument, B) \in C(\hat E)$ for all $B \in \calB(\hat E)$. 
\end{lem}

\begin{thm}\label{t.compactc0}
    Let $E$ be locally compact and $T \in \calL(C_0(E))$. Then the following two characterizations hold: 
    \begin{enumerate}[\upshape (a)]
        \item \label{i.t.compactc0.WeakComp}
              The following are equivalent:
              \begin{enumerate}[\upshape (i)]
                  \item \label{i.i.t.compactc0.WeakComp.wc}
                        $T$ is weakly compact.
                  
                  \item \label{i.i.t.compactc0.WeakComp.kernelCond}
                        $T$ is a kernel operator and its associated kernel $k$ satisfies $k(\argument, A) \in C_0(E)$ for all $A\in \calB(E)$.
              \end{enumerate}

        \item \label{i.t.compactc0.Comp}
              The following are equivalent:
              \begin{enumerate}[\upshape (i)]
                  \item \label{i.i.t.compactc0.Comp.c}
                        $T$ is compact.
                  \item \label{i.i.t.compactc0.Comp.kernelCond}
                        $T$ is a kernel operator and for the associated kernel $k$ one has that $x\mapsto k(x, \argument)$ is continuous
                        in the total variation norm and the map $x\mapsto \|k(x, \argument)\|_{\mathrm{TV}}$ belongs to $C_0(E)$.
              \end{enumerate}
    \end{enumerate}
\end{thm}

\begin{proof}
    \ref{i.t.compactc0.WeakComp}:
    \impliesProof{i.i.t.compactc0.WeakComp.wc}{i.i.t.compactc0.WeakComp.kernelCond} This follows along an analogous line of arguments as in the proof of Theorem~\ref{t.strongfellerop}: We consider the (norm-)adjoint $S=T^*\in \calL(\calM(E))$ and set $k(x, \argument) \coloneqq S\delta_x$. Then $\|k\| = \|T\| < \infty$. Moreover, as $T$ is $\sigma(C_0(E), \calM(E))$-compact, it follows from Proposition~\ref{p.weakcomp} that
    \[
        k(\argument, A) = S^*\one_A \in C_0(E) \qquad \text{for all } A \in \calB(E).
    \]

    \impliesProof{i.i.t.compactc0.WeakComp.kernelCond}{i.i.t.compactc0.WeakComp.wc}
    Let $\hat k \colon \hat E \times \calB(\hat E) \to \bbK$ be the kernel defined in Lemma~\ref{l.kernel.compactification}.
    By Lemma~\ref{l.kerneloperator}, there exists a bounded operator $T_k \colon B_b(\hat E) \to B_b(\hat E)$ such that $T_k(\one_A) = k( \argument, A)$ for all $A \in \calB(\hat E)$. Set $\hat T \coloneqq T_k\rvert_{C(\hat E)}$ and let
    \begin{align*}
        Z \coloneqq \set{f \in C(\hat E) : f(\infty) = 0}.
    \end{align*}
    Then, as a consequence of Theorem~\ref{t.strongfellerop}, $\hat T$ is weakly compact and $\hat T C(\hat E) \subseteq Z$. Moreover, $T = \rho \hat T \iota$, where $\rho \colon Z \to C_0(E)$ and $\iota \colon C_0(E) \to C(\hat E)$ denote the canonic restriction and extension operators, respectively. Hence, $T$ is weakly compact. \smallskip

    \ref{i.t.compactc0.Comp}:
    \impliesProof{i.i.t.compactc0.Comp.c}{i.i.t.compactc0.Comp.kernelCond}
    We use the notation introduced in part \ref{i.t.compactc0.WeakComp} of the proof: As $T$ is compact, it is also weakly compact and thus \ref{i.t.compactc0.WeakComp} implies that $T$ is a kernel operator and its associated kernel $k$ satisfies $k(\argument, A) \in C_0(E)$ for all $A\in \calB(E)$. Moreover, it follows that $\hat T$ is compact, since $T$ is compact and $Z$ is a closed subspace of $C(\hat E)$. Thus, Theorem~\ref{t.ultrafellerop} implies that the kernel $\hat k$ of $\hat T$ is ultra Feller. Hence, $k$ is ultra Feller, too, and it is easy to see that the map $x\mapsto \|k(x, \argument)\|_{\mathrm{TV}}$ belongs to $C_0(E)$. \smallskip

    \impliesProof{i.i.t.compactc0.Comp.kernelCond}{i.i.t.compactc0.Comp.c}
    Extending functions once again to $\hat E$, \ref{i.i.t.compactc0.Comp.kernelCond} in combination with Theorem~\ref{t.ultrafellerop} implies that
    $\hat T$ is a $\beta_0(\hat E)$-compact operator on $C(\hat E)$. As $\hat{E}$ is compact, $\beta_0(\hat{E})$ coincides with the norm topology, which shows that $\hat T$ is indeed compact on $C(\hat E)$. Thus, $T = \rho \hat T \iota$ is compact on $C_0(E)$. 
\end{proof}

Theorem~\ref{t.compactc0}\ref{i.t.compactc0.WeakComp} has the following surprising consequence.
\begin{cor}
    Let $T,S$ be bounded linear operators on $C_0(E)$ with kernels $k_T$ and $k_S$, respectively. Suppose that the following two assertions hold:
    \begin{enumerate}[\upshape (a)]
        \item $0 \le S \le T$, i.e., $0 \le k_S(x,A) \le k_T(x,A)$ for all $x \in E$ and $A \in \calB(E)$.
        \item $k_T(\argument,A) \in C_0(E)$ for each $A \in \calB(E)$.
    \end{enumerate}
    Then $k_S(\argument,A) \in C_0(E)$ for each $A \in \calB(E)$.
\end{cor}

\begin{proof}
    Theorem~\ref{t.compactc0}\ref{i.t.compactc0.WeakComp} implies that $T$ is weakly compact. Since the dual space of $C_0(E)$ is an AL-space by \cite[Theorem~9.27]{Aliprantis2006}, it has order continuous norm (cf.~the argument given below \cite[Theorem~9.33]{Aliprantis2006}) and, hence, weak compactness of positive operators on $C_0(E)$ is inherited under domination according to \cite[Theorem~3.5.11(3)]{MeyerNieberg1991}. So $S$ is also weakly compact, and another application of Theorem~\ref{t.compactc0}\ref{i.t.compactc0.WeakComp} yields $k_S(\argument ,A) \in C_0(E)$ for all $A \in \calB(E)$.
\end{proof}

\begin{rem}
    It follows from Theorem~\ref{t.compactc0}, that an operator $T$ on $C_0(E)$ is weakly compact if and only if
    $T^{**}$ maps $B_b(E)$ to $C_0(E)$. We note that if we know $T$ to satisfy the strong Feller property and 
    if $T$ is positive, then $T^{**}$ maps $B_b(E)$ to $C_0(E)$ if and only if it maps the constant $\one$ function
    to $C_0(E)$. This criterion is well known for semigroups governing parabolic PDE with unbounded coefficients, see
    \cite[Theorem 4.1.10]{Lorenzi2017} but, as Theorem~\ref{t.compactc0} shows, it holds in greater generality.
\end{rem}
\appendix

\section{\texorpdfstring{$\sigma(\calM(E), C_b(E))$}{σ(M(E), Cb(E))}-separability of \texorpdfstring{$\calM(E)$}{M(E)}}
\label{apx:separability}

In Section~\ref{sect.classic}, the assumption that there exists a sequence of measures in 
$\calM(E)$ that separates the points in $C_b(E)$ is used and does not seem very intuitive. It is certainly fulfilled if the space $E$ is separable.
In the following lemma, we clarify the connection of the condition that $\calM(E)$ separates the points in $C_b(E)$ to the separability of $E$ and also to the so-called countable chain condition.
We recall that a topological space $E$ satisfies the \emph{countable chain condition} if every family of disjoint, non-empty open subsets of $E$ is at most countable. 

We point out that the following lemma can be obtained from known results. In particular, the equivalence of~\ref{i.l.separability.E-separable} and~\ref{i.l.separability.ccc}
for metric spaces follows from \cite[Corollary~4.1.16]{Engelking1989} and the proof of~\ref{i.l.separability.calME-separable}~$\Rightarrow$~\ref{i.l.separability.ccc} follows along the lines of 
the proof of Proposition~3.1 and Theorem 4.9 in~\cite{Kandic2018}. However, in an effort to being self-contained, we have decided to include a complete and elementary proof.

\begin{lem}
    \label{l.separability}
    Let $E$ be a completely regular Hausdorff topological space. Consider the following assertions:
    \begin{enumerate}[\upshape (i)]
        \item\label{i.l.separability.E-separable}
              $E$ is separable.

        \item\label{i.l.separability.calME-separable}
              $\calM(E)$ is $\sigma(\calM(E), C_b(E))$-separable.

        \item\label{i.l.separability.ccc}
              $E$ satisfies the countable chain condition.
    \end{enumerate}
    Then {\upshape(i)} $\Rightarrow$ {\upshape(ii)} $\Rightarrow$ {\upshape(iii)}. In the case that $E$ is metrizable, then also {\upshape(iii)} $\Rightarrow$ {\upshape(i)}. Thus, in this case, all the three assertions above are equivalent.
\end{lem}
\begin{proof}
    \impliesProof{i.l.separability.E-separable}{i.l.separability.calME-separable}
    If $E$ is a separable, we let $(x_n)_n$ be a dense sequence in $E$. Then the sequence $(\delta_{x_n})_n$ of Dirac measures clearly separates the elements of $C_b(E)$. \smallskip

    \impliesProof{i.l.separability.calME-separable}{i.l.separability.ccc}
    First, we notice that the $\sigma(\calM(E), C_b(E))$-separability of $\calM(E)$ implies the existence of a strictly positive functional on $C_b(E)$:
    Indeed, due to the $\sigma(\calM(E), C_b(E))$-separability, there exists a sequence $(\mu_n)_n$ in $\calM(E)$ that separates the points in $C_b(E)$. We may assume that $\lVert \mu_n \rVert_{\mathrm{TV}} \leq 1$ for all $n \in \bbN$. We define
    \begin{align*}
        \varphi: C_b(E) \to \bbK, \qquad \varphi \coloneqq \sum_{n \in \bbN} 2^{-n} \lvert \mu_n \rvert.
    \end{align*}
    Since $(\mu_n)_n$ separates points, it follows that $\varphi(f) > 0$ for all non-zero and positive functions $f \in C_b(E)$.

    Now let $\calU$ be any family of open and pairwise disjoint subsets of $X$. As $E$ is completely regular, for each $U \in \calU$ there is a positive continuous function $f_U \in C_b(E)$ , which vanishes outside $U$ and has norm $1$. Since $f_{U_1} \wedge f_{U_2} = 0$ for distinct $U_1, U_2 \in \mathcal{U}$, it follows that $\left\lVert \sum_{k = 1}^n f_{U_k} \right\rVert = 1$ for every pairwise distinct $U_1, \dots, U_n \in \calU$ and $n \in \bbN$. Assuming that $\varphi \in \calM(E)'$ is strictly positive and has norm $1$, it follows that the set
    \begin{align*}
        \calU_n \coloneqq \{ U \in \calU : \varphi(f_U) \geq 1/n \}
    \end{align*}
    has cardinality less or equal to $n$. Hence, $\calU = \bigcap_{n \in \bbN} \calU_n$ is countable and $E$ satisfies the countable chain condition. \smallskip

    \impliesProof{i.l.separability.ccc}{i.l.separability.E-separable}
    Without loss of generality, assume $E$ is uncountable; otherwise there is nothing to show.
    Aiming for a contradiction, assume that $E$ is not separable. Let $\ud$ be a metric on $E$, fix $\delta > 0$ and set 
    \begin{align*}
        \calP_\delta \coloneqq \{P \subset E : \forall x, y \in P \colon x \neq y \Rightarrow \ud(x, y) \geq \delta \} \subset 2^E. 
    \end{align*}
    Ordered by inclusion, the set $\calP_\delta$ is a partially ordered set. Moreover, it is straightforward to check that every chain
    in $\calP_\delta$ has a maximal element. Hence, by Zorn's lemma, $\calP_\delta$ contains a maximal element $P_\delta$. Now consider the union 
    \begin{align*}
        P \coloneqq \bigcup_{n \in \N} P_{\nicefrac{1}{n}}
    \end{align*}
    and notice that $P$ is dense in $E$. Indeed, for every $n \in \N$ and every point $y \in E$, it follows that there exists $x \in P_{\nicefrac{1}{n}}$ such that $\ud(x, y) < \tfrac{1}{n}$. Otherwise, we may add $y$ to $P_{\nicefrac{1}{n}}$, ending up with a larger set in $\calP_{\nicefrac{1}{n}}$ than $P_{\nicefrac{1}{n}}$. This contradicts the maximality of $P_{\nicefrac{1}{n}}$.
    
    Since we assume the non-separability of $E$ and $P$ is a countable union of sets, there must necessarily be an $n \in \N$ such that $P_{\nicefrac{1}{n}}$ is uncountable. Then the collection of open balls with $\Ball(x, \tfrac{1}{3n})$ with radius $\tfrac{1}{3n}$ centred at $x \in P_{\nicefrac{1}{3n}}$ contains uncountably many open disjoint sets. This contradicts the countable chain condition.
\end{proof}

An example of a non-separable compact Hausdorff topological space that satisfies the countable chain condition can be found in \cite[Example~24]{Steen1995}. In particular, the implication \ref{i.l.separability.ccc} $\Rightarrow$ \ref{i.l.separability.E-separable} in Lemma \ref{l.separability} is not true in general.

\subsection*{Acknowledgements} The authors thank Jochen Gl\"uck for countless discussions on the content of this article. Moreover, we thank the anonymous referee for many constructive comments that helped improve this article.

\end{document}